\documentclass{amsproc}
\usepackage{epsfig,latexsym,amsfonts,amssymb,amsmath,amscd}

\newtheorem{theorem}{Theorem}[section]
\newtheorem{lemma}[theorem]{Lemma}
\newtheorem{lem}[theorem]{Lemma}

\theoremstyle{definition}
\newtheorem{definition}[theorem]{Definition}
\newtheorem{example}[theorem]{Example}

\newtheorem{proposition}[theorem]{Proposition}

\theoremstyle{remark}

\newtheorem{rem}[theorem]{Remark}
\newtheorem{ques}[theorem]{Question}
\numberwithin{equation}{section}

\newcommand{\abs}[1]{\lvert#1\rvert}
\newtheorem{MNote}[theorem]{Major Note}

\def\mcA{\mathcal A}

\def\prect{{\,\preceq}_{\operatorname{func}}\,}

\def\mcH{\mathcal H}

\def\mcG{\mathcal G}
\def\mcL{\mathcal L}

\def\mcM{\mathcal M}

\def\R{\mathbb R}
\def\nsets{\mathcal P^*}

\newcommand{\etype}[1]{\renewcommand{\labelenumi}{(#1{enumi})}}
\def\eroman{\etype{\roman}}

\newcommand{\Net}{\mathbb N}
\newcommand{\Q}{\mathbb Q}
\newcommand{\Z}{\mathbb Z}

\def\supp{\mathrm{supp}\,}

\newcommand{\trop}[1]{\mathcal{#1}}

\newcommand{\tG}{\trop{G}}

\newcommand{\one}{\mathbf 1}
\newcommand{\zero}{\mathbf 0}
\newcommand{\Hzero}{\mathbf 0_\mathbf H}

\newcommand{\tT}{\trop{T}}
\def\tTz{\trop{T}_0}

\def\sg{\sigma}

\def\bfv{{\mathbf v}}
\def\bfw{{\mathbf w}}

\def\ctw{\cdot_{\operatorname{tw}}}

\def\ttT{ {\tT}_{\sqrt{a}}}

\def\ZZ{\mathbb Z}

\def\cocoa{{\hbox{\rm C\kern-.13em o\kern-.07em C\kern-.13em o\kern-.15em A}}}

\def\la{\lambda}

\def\w2M{\bigwedge^2M}

\def\w{\wedge }

\def\Z{\ZZ}
\protect
\protect
\protect \protect

\protect
\protect

\def\be{\begin{equation}}
\def\ee{\end{equation}}
\def\bclm{\begin{claim}}
\def\eclm{\end{claim}}
\def\beqn{\begin{eqnarray}}
\def\eeqn{\end{eqnarray}}
\def\beqn*{\begin{eqnarray*}}
\def\eeqn*{\end{eqnarray*}}

\numberwithin{equation}{section}
\begin{document}

\title{Roots of polynomials over semirings and hyperfields}

\author{Louis Halle Rowen}

\address{Department of Mathematics, Bar-Ilan University, Ramat-Gan 52900,
Israel}

\email{rowen@math.biu.ac.il}


\subjclass[2020]{Primary 08A40;  14T10; 16Y20; 16Y60; Secondary:   12F05;  12K10; 15A78;    15A80.}
\date{July 27, 2026}


\keywords{algebraic, doubled pair, extension, factor, hyperfield, hyperpair, integral,  metatangible, pair,  polynomial,  polynomial function, tangible polynomial, semiring,
 supertropical, system, triple, tropical extension, root.}

 \begin{abstract}    We lay the groundwork for the theory of roots of polynomials over semirings and hyperfields, employing a property of a ``surpassing relation'' $\preceq$ on semirings and hyperfields that we call strong $\preceq$-reversibility.
 There are three kinds of roots generalizing roots in classical algebraic theory  --- ``null roots,''  ``$\preceq$-roots,'' and ``factor-roots.'' The theory works best when all null roots are also factor-roots. The ensuing results include  the fundamental theorem of algebra for pairs, asserting that    tangible polynomials with enough roots ``$\preceq$-split,'' sometimes uniquely, into linear factors. We also examine to what extent polynomials are determined by their null roots.

Next, we  obtain $\preceq$-roots over extension systems,  and construct integrally closed systems over various base systems, including hyperfield systems over a group. The point is that whereas hyper-polynomials behave poorly,  polynomial systems remain strongly $\preceq$-reversible, so one can compose chains of $\preceq$-finitely spanned extensions.

Next, we introduce real pairs, focusing on  square roots and quadratic closure.

In the appendix we recast the theory by considering equivalence classes of tangible polynomials which agree on almost all substitutions.
\end{abstract}
\maketitle

\tableofcontents




\section{Overview} A few years ago, Baker and  Lorscheid \cite{BaL} re-derived Descartes’ rule of signs by studying polynomials over the sign hyperfield.
This has led to increased interest in roots of polynomials over hyperfields, including alternative  formulations of \cite{BaL} in \cite{gu} for ``idylls,'' and later in \cite{AGT} for tropical extensions. On the other hand, Descartes’ original rule of signs relies on the property that any polynomial of degree $n$ has at most $n$ roots, whereas there are polynomials over the phase hyperfield that have infinitely many roots, as noted in~\cite{BaL}. Since \cite{AGT} was developed in the more inclusive language of ``systems,''   recently generalized in \cite{JMR3, AGR2} to ``pairs''\footnote{A pair $(\mcA,\mcA_0)$  over a monoid $\tT$ is a $\tT$-module $\mcA$ with a designated ``null'' $\tT$-submodule~$\mcA_0$.}, we are motivated here to develop the theory of roots of polynomials in this context.  Namely, we continue the study of roots from \cite{Row26}.

It turns out  that there are several competing definitions paralleling classical algebra for what it means for $a\in \tT$ to be a root of a tangible polynomial $f$ (i.e., whose coefficients are in $\tT$), one involving null evaluations (Definition~\ref{subs}), one involving the surpassing relation, and the third, ``factor-root,'' involving factorization by $\la (-) a$ (Definition~\ref{proot}). In most cases,  factor-roots are $\preceq$-roots.
  These definitions jibe in   Theorem~\ref{Rev}, under the hypothesis of
  ``fissure'' (Definition~\ref{fis0}).   

  The ensuing theory includes Theorem~\ref{spl2}, which states that any tangible polynomial $f$ with ``enough''  factor-roots  $\preceq$-splits as $f \preceq \prod_{i=1}^n(\lambda(-)a_i)$, where the right side is a   product of tangible binomials. A more subtle concept is that $f$ is $\prect$-split if $f(a) \preceq \prod_{i=1}^n(a(-)a_i)$ for almost all $a\in \tT.$ In the classical case this would say that $a_1,\dots, a_n$ include all the roots of $f,$ but the roots of  $\prect$-split polynomials remain a mystery to us in this generality.

   We also investigate simultaneous roots and multiple roots, and ``splittings'' of a polynomial.
Then we turn to the question, ``To what extent is a polynomial    determined by its roots?''
We must view a polynomial as a function, since  its null roots are determined by the function, and two different abstract polynomials could define the same function. Taking $\tT$ infinite, we define an equivalence $g_1\equiv g_2$ for  polynomials     $ g_1 $ and $ g_2 $ if  $g_1(a)= g_2(a) $ for almost all $a\in \tT$.  The pair $(\mcA,\mcA_0)$  is called \textbf{$\tT$-ubiquitous} if equivalent tangible polynomials always are equivalent to a common sub-polynomial. By Theorem~\ref{dens},  metatangible pairs are $\tT$-ubiquitous; i.e., any tangible polynomial $f$ has an equivalent polynomial $h \prect \prod_{i=1}^t (\la (-)a_i)^{m_i} g$  with $g$ unique up to equivalence.

 We turn to the process of adjoining a root of a polynomial $f$ by extending the pair. We can do this whenever $f$ is not $\prect$-split (Theorems~\ref{pol1a1} and~\ref{pol1a2}). Unlike some other related conditions, ``strong reversibility'' is preserved under  these extensions, thereby enabling one to adjoin roots in towers of  extensions. Continuing the process yields a system over which every non-$\prect$-split tangible polynomial has a $\preceq$-root (Theorem~\ref{FT1},  a weak fundamental   theorem of algebra for systems).

As a special case, we present the theory in \S\ref{real0} for ``real pairs,'' where $\tT = \tT_+ \cup (-)\tT_+,$ with $\tT_+$   a submonoid and $\tT_+ \cap (-)\tT_+ = \emptyset.$

To close the circle, in forthcoming  work,  we    examine ``real'' roots and the connection to Rolle's theorem and Descartes’ rule of signs. This  involves more technical concepts which would make this paper unwieldy.

\subsection{Preliminaries}$ $

We review the preliminaries; the reader may turn to \cite{AGR2,Row26} for more details.

 An \textbf{additive semigroup} is a commutative semigroup, with   operation  ``$+$'' and neutral element~$\zero.$ 

A \textbf{semiring}  $(\mcA,+,\cdot,\zero,\one)$ is an additive semigroup   $(\mcA, +,\zero)$ endowed  with multiplication~$\cdot$ that distributes \footnote{At times we want to forego distributivity. We shall return to this issue when discussing hyperrings and also in the appendix.} over addition $+$, and having a distinguished element $ \one \ne \zero$ such that $(\mcA, \cdot,\one)$  is a monoid, with $\zero$ multiplicatively absorbing, in the sense that $\zero \cdot a = a \cdot \zero = \zero$ for all $a\in \mcA.$\footnote{If need be, one easily can  adjoin $\zero$ formally. There  is also a construction for adjoining~$\one$, given in  \cite[p.~3]{golan92},   a standard reference for semirings.}

  A semiring $\mcA$ is \textbf{idempotent} (resp.~\textbf{bipotent}) if  $a_1+a_1=a_1$  (resp.~$a_1+a_2 \in \{a_1,a_2\}$)  for any $a_1,a_2 $
in~$\mcA.$
    Any ordered monoid $(\tG,\cdot)$ gives rise to a bipotent semiring, where $a_1+a_2 = \max\{a_1,a_2\}$.

  To overcome the lack of negation in various algebraic structures,  blueprints were introduced in \cite{Lor1,Lor2}, and were placed in a more general context in~2016 in~\cite{Row22}, in terms of a  \textbf{negation map} $(-)$ which generalizes classical negation, and a \textbf{surpassing relation}; together these constitute a  \textbf{system}, as initiated in \cite{Row22}, cf.~Definition~\ref{sy}.

  Examples treated in~\cite{AGR2}  include supertropical pairs, tropical extensions of  pairs, doubled (symmetrized) pairs, hyperfield pairs,   and polynomial pairs.
\subsubsection{Modules over a  monoid}$ $

Just as one studies rings as algebras over a given commutative ring (often a field), we might be inclined to study
semirings over a given underlying commutative semiring, but it turns out that the key is the underlying multiplicative structure, which is a monoid.

\begin{definition} In this paper, $\tT$ always denotes a commutative monoid with unit element~$\one$.

\begin{enumerate}
  \item
A $\tT$-\textbf{module}  
is an additive semigroup
$(\mathcal A,+,\zero_\mcA)$ together with a  (left)   $\tT$-action $\tT\times \mathcal A \to \mathcal A$ (denoted  as
concatenation), which is
\begin{enumerate}
\item \textbf{associative},  in the sense that
$a_1(a_2b) = (a_1a_2)b $ for all $ a_i \in \tT,$ $b  \in \mathcal A$. (Hence $a_1(a_2b)=a_2(a_1b).$)

 \item zero-absorbing, i.e., $a \zero _\mcA = \zero_\mcA \  \text{for all}\; a \in \tT.$

  \item \textbf{distributive},  in the sense that
$$a(b_1+b_2) = ab_1 +ab_2,\quad \text{for all}\; a \in \tT,\; b_i \in \mathcal A.$$

\item  $\one b= b$ for all $b\in \mcA.$
 \end{enumerate}

  \item We call $\tT$ the \textbf{underlying monoid}  of  \textbf{tangible} elements.
  Our convention is to write $a$ for an element of $\tT,$ and $b$ for an element of $\mcA.$

\end{enumerate}
\end{definition}

\begin{MNote} \label{warn0}

       We shall assume henceforth the following properties for a $\tT$-module $\mcA$.
 \begin{enumerate}\eroman    \item     $\tT\subseteq \mcA$ (in which case $\mcA$ is called \textbf{weakly admissible} in \cite{AGR2}), and we write $\tTz := \tT \cup \{\zero\}, $ also a monoid.
 We view $\mcA$ as a $\tTz$-module by declaring $\zero b= \zero$ for all $b\in \mcA$.

   \item Every element of $\mcA$ is a sum of elements of $\tTz.$\footnote{If necessary, one could replace $\mcA$ by the submodule $\overline{\mcA}$ spanned by $\tTz.$ If $\mcA$ is a semiring then $\overline{\mcA}$ is also  a semiring. However, bear in mind that if $\mcA$ has multiplication which is not distributive over addition, cf.~footnote to~Remark~\ref{warn}, $\overline{\mcA}$ need not be closed under multiplication.}

     \item    $\mcA $ is torsion-free over $\tT$, in the sense that
if $a b _1= ab_2$, for  $a\in \tT$ and $b_1,b_2\in \mcA$, then  $b_1=b_2.$

 \item   When   $\mcA $
 has a binary multiplication, we assume that  $\one_\tT = \one_\mcA$, which we write as $\one$.
    We also assume  that $ab = b  a $  and $a(b_1 b_2) = (a b_1)b_2=b_1(a b_2) $ for all $a\in \tT,$ $b_i\in \mcA,$ i.e., the copy of $\tT$ in $\mcA$ is central.
    \end{enumerate}
\end{MNote}


  \subsection{Pairs, triples, and systems} $ $

  The venue for the theory is ``pairs,''  ``triples,'' and ``systems'' which we review briefly.
  The underlying idea, originating in Gaubert's dissertation~\cite{Gau}, is to replace the zero element of a semiring, which often is useless, by a designated additive $\tT$-submodule.

 \begin{definition}\label{symsyst}
We use the following notions \cite{JMR3, AGR2, Row26} to compensate for the lack of negatives.
\begin{enumerate}
    \item A  $\tT$-\textbf{pair}  $(\mcA,\mcA_0)$ is  a  $\tT$-module $\mcA$, given together with a specified $\tT$-submodule  $\mcA_0$, i.e., $ab_0  \in \mcA_0$ for all $a\in \tT$ and $b_0\in \mcA_0,$ and which also satisfies the converse:

If $a b \in \mcA_0$ for $a\in \tT$ and  $b\in \mcA$, then $b\in \mcA_0$.

 \item We simply call a $\tT$-pair     a ``pair.''
  $\mcA_0$ is called the \textbf{null submodule} of~$\mcA.$


   \item A \textbf{negation map} on a  pair
$(\mcA,\mcA_0)$ is a
 semigroup isomorphism
$(-) :\mathcal A \to \mathcal A$ written
$a\mapsto (-)a$, of order~$\le 2,$  satisfying the following properties:
   \begin{itemize}
          \item      $(-)\one \in \tT.$
          \item  $((-)a)b = (-)(ab) = a((-)b)$ for $a,b \in \mathcal A.$
         \end{itemize}
         We write  $b_1(-)b_2$ for $b_1+((-)b_2)$, $b^\circ $ for $ b(-)b,$ $e$ for $\one^\circ$, and $b_1(\pm)b_2$ for either $b_1+b_2$ or $b_1(-)b_2$. The pair is \textbf{uniquely negated}  if $a+\one \in \mcA_0$ for $a\in \tT$ implies $a=(-)\one$.

    \item   A \textbf{pre-triple} $(\mcA,\mcA_0,(-))$      \cite{Row22} is a pair with a negation map. A \textbf{triple} $(\mcA,\mcA_0,(-))$   is a uniquely negated pre-triple.

       \item $(\mcA,\mcA_0)$ is of the \textbf{first kind} if $\one + \one =e,$ i.e., $(-)\one = \one;$ otherwise   $(\mcA,\mcA_0)$ is of the \textbf{second kind}.
  \end{enumerate}
\end{definition}

\begin{rem}\label{un}
    When $(\mathcal A,\mcA_0)$  is uniquely negated, if $a_2 = a_1+b $ for $a_i\in \tT$ and~$b\in \mcA_0,$ then $a_1=a_2$, since $a_2(-)a_1 = a_1^\circ + b\in \mcA_0.$
\end{rem}
\begin{definition}\label{srp}$ $
    \begin{enumerate}
    \item
    The pre-triple  $(\mcA,\mcA_0,(-))$   is a \textbf{$(-)$-domain} if  $a_1b(-)a_2b\in \mcA_0$ for $a_1\ne a_2 \in \tT$, $b\in \mcA,$ implies $b \in \mcA_0$.

    \item
A \textbf{semiring pre-triple} $(\mcA,\mcA_0,(-)) $ is a pre-triple for which $\mcA$ is a semiring.

  \item A \textbf{semiring $(-)$-domain} is a semiring pre-triple which is a $(-)$-domain.

    \item
    A pair  $(\mcA,\mcA_0)$ endowed with binary multiplication  is a   \textbf{strong domain} if  $b_1b_2\in \mcA_0$ for $b_1,b_2\in \mcA,$ implies $b_1 \in \mcA_0 $ or $b_2 \in \mcA_0$.

 \end{enumerate}
\end{definition}

\begin{rem}\label{TA} $ $
    \begin{itemize}
    \item We assume throughout that $\mcA_0\ne \mcA$. This implies that $ \tTz \cap \mcA_0 = \{\zero\},$ since otherwise if $a \in \tT \cap \mcA_0 $ and $b\in \mcA$ then $ab \in \mcA_0,$ so $b\in \mcA_0.$
        \item A semiring pre-triple $(\mcA,
        \mcA_0,(-))$ is a semiring $(-)$-domain if and only if $(a_1(-)a_2)b\in \mcA_0$ implies $b\in \mcA_0.$

 \item  By \cite[Lemma~2.27]{Row26}, every   semiring pair satisfying $ \tT \cup \mcA_0 =\mcA  $     is a strong domain.

    \end{itemize}
\end{rem}

\begin{example}\label{tr1} Two easy examples for intuition:
    \begin{enumerate}\eroman
\item \textbf{``Classical algebra''} is when $\mcA$ is an algebra over a commutative subring~$C$,  where $\tTz =C$, viewed as a multiplicative monoid, and $\mcA_0 =\{\zero\}$.
Our objective is to build a theory which contains much of classical algebra.

   \item The ``trivial'' pair $(\mathbb T, \zero)$ where $(\mathbb T,+) = \{\zero,\one\},$ with $\one +\one = \zero$, and underlying monoid $\{\one\}$.
    \end{enumerate}
\end{example}

  Other examples are given in \cite[\S 5]{AGR2} and  in \cite[Example 2.28 and \S2.2]{Row26}.

\medskip  \subsubsection{Surpassing relations and reversibility}$ $

We  need some relation generalizing equality. Lorscheid~\cite{Lor1} used a symmetric relation; the relation that we adapt from \cite{Row22}, on the contrary,  often  is antisymmetric.

 \begin{definition}\label{surp} Suppose  $\mathcal A$ is a  $\tT$-module.
\begin{enumerate}\eroman
 \item   A $\tT$-\textbf{preorder}  $\preceq$  on $\mathcal A$  is a set-theoretic preorder  satisfying the following algebraic conditions,  for all   $b,b_i,b_i' \in \mcA$:
  \begin{enumerate}
      \item  For all $a\in \tT$, $b_1\preceq b_2$ if and only if $a b_1\preceq a b_2 .$
         \item   $b_i \preceq b_i'$ implies $b_1 + b_2 \preceq b_1'+b_2'.$
    \end{enumerate}
 \item  A \textbf{pre-surpassing relation} on a pair $(\mathcal A,\mcA_0)$  is a $\tT$-preorder $\preceq$ satisfying: \begin{itemize}
\item   If
$b\in \mcA_0$ then $ \zero \preceq b. $  \end{itemize}\item  A \textbf{surpassing relation} on a pair $(\mathcal A,\mcA_0)$  is a pre-surpassing relation   $\preceq$ on~$\mcA$ satisfying:
\begin{itemize}\label{ta1}
\item
$a_1 \preceq a_2 $ for $a_1 ,a_2 \in   \tTz$ implies $a_1 =a_2.$
 \end{itemize}
\end{enumerate}

 \end{definition}

\begin{rem}
    If $(\mcA,\mcA_0)$ has a surpassing relation then   $\mcA_0\cap \tTz =\{ \zero\}$, by \cite[Lemma~2.15]{Row26}.
\end{rem}

\begin{example}
    Define $b_1 \preceq_0 b_2$ if $b_2 = b_1 +c$ for $c\in \mcA_0.$
\end{example}
\begin{lem}
    \label{adom}
Any $(-)$-domain satisfies the property that if  $a_1b \preceq_0 a_2b$ for $a_1\ne a_2 \in \tT$, $b\in \mcA,$ then $b \in \mcA_0$.
 \begin{proof}
      $a_1 b  \preceq_0 a_2 b$ for $a_1\ne a_2$ implies $a_2b (-) a_1 b\in \mcA_0,$ and thus $b\in \mcA_0.$
 \end{proof}
\end{lem}

\begin{definition}\label{sy}
    A \textbf{pre-system} $(\mcA,\mcA_0,(-),\preceq)$ is a pre-triple with a pre-surpassing relation $\preceq$.

     A~\textbf{system}  $(\mcA,\mcA_0,(-),\preceq)$ is  a triple $(\mcA,\mcA_0,(-))$ with a surpassing relation $\preceq$, cf.~\cite{Row22}.
\end{definition}

   The classical example is an integral domain $\mcA$, with underlying monoid $\tT:=\mcA \setminus \{\zero\},$ where $\mcA_0 = \{\zero\},$ $(-)$ is classical negation, and $\preceq$ is equality.

 \begin{lem}\label{sur1} $ $
     \begin{enumerate}\eroman
         \item   The pre-surpassing relation $\preceq_0$ on a  pair $(\mcA,\mcA_0)$ is given by $b_1\preceq_0 b_2$ if $b_2 = b_1+b_0$ for some $b_0\in \mcA_0$.

   \item In  the other direction, given  a $\tT$-module $\mcA$ with a preorder $\preceq$, let $\mcA_0 = \{ b\in \mcA : \zero \preceq b\}.$ Then $(\mcA,\mcA_0)$ is a pair, and  $\preceq_0$ is a pre-surpassing relation, which is a surpassing relation if $(\mcA,\mcA_0)$ is uniquely negated.
     \end{enumerate}
 \end{lem}
 \begin{proof}
(i) is
\cite[Lemma~3.28]{AGR2}.

(ii) The condition of  Definition~\ref{surp}(ii) holds by definition. If $a_1 +b = a_2$ for $a_i\in \tT$ and $b\in \mcA_0,$ then  $a_2 = a_1$ by Remark~\ref{un}.
 \end{proof}

 Thus, on a triple, $\preceq_0$ is a  surpassing relation.

\subsection{The main concepts in this paper}$ $

%

\begin{definition}
      A  relation $\preceq$ on a pre-triple $(\mathcal A,\mcA_0,(-))$  is called: \begin{enumerate}
    \item  \textbf{  $\tT$-reversible} if  $(-)a_0 \preceq \sum_{i=1}^n a_i$   implies $(-)a_1 \preceq  a_0 +\sum_{i=2}^n a_i$,

 \item \textbf{strongly $\tT$-reversible} if  $a+b\in \mcA_0$  implies $ (-)a  \preceq b$,
\end{enumerate}
for $a, a_i\in\tT$, $b\in \mcA$.
\end{definition}

\begin{lem}[{\cite[Lemma~2.17]{Row26}}]\label{rever1}
 Any strongly $\tT$-reversible pre-surpassing relation  on a pre-triple  is $\tT$-reversible, and then $(\mathcal A,\mcA_0,(-),\preceq)$ is a system.
  \end{lem}

\begin{definition}
    \label{fis0} (modified from \cite{gu}) We say that  $\preceq$ satisfies  \textbf{fissure}  if whenever $a_0 \preceq \sum_{i=1}^n a_i \in \mcA_0$ for $a_i\in \tTz$, there is $a\in \tTz$ such that $a_0 \preceq a_1+a   $ and $ a \preceq
  \sum _{i=2}^n  a_i.$
\end{definition}

\begin{lem}[{\cite[Lemma~2.18]{Row26}}]\label{fis}  Any surpassing relation $\preceq$ satisfying  fissure on  a triple   is   strongly $\tT$-reversible.
\end{lem}

The main example of fissure is the hyperpair, cf.~Remark~\ref{warn} below.

\subsubsection{Metatangible pairs}$ $

The following kind of pair is especially malleable.

 \begin{definition}\label{metat}$ $
 The pair $(\mcA,\mcA_0)$ is  \textbf{metatangible} if $a_1+a_2 \in \tT \cup \mcA_0$ for any $a_1,a_2 \in \tT$.  \end{definition}

Although, as we see below in \S\ref{expa}, there are many examples of metatangible pairs which are far from classical, the tangible elements behave ``classically,'' as we now see.

 \begin{lem}
    \label{sub1}  Suppose $a_1+a_2 = a_3$ in a system.
    \begin{enumerate}\eroman
        \item  Then $a_2 \preceq a_3(-)a_1$.     \item   If  $(\mcA,\mcA_0)$   is metatangible and $a_1,a_3\in \tT$  with  $a_1\ne a_3$, then $a_2 = a_3(-)a_1$.
    \end{enumerate}
\end{lem}
\begin{proof} (i)
$a_2 \preceq a_2 +a_1 (-)a_1 = a_3(-)a_1.$

(ii) $a_3(-)a_1$ is in $\tT,$ so equals $a_2,$ by (i).
\end{proof}

\begin{lem}
    Any uniquely negated metatangible semiring triple is a $(-)$-domain.
\end{lem}
\begin{proof}
    Suppose $a_1b(-)a_2b \in \mcA_0$ for $a_1\ne a_2$ in $\tT.$ Then $a_1(-)a_2 \in \tT,$ and $(a_1(-)a_2)b \in  \mcA_0,$ so $b\in \mcA_0.$
\end{proof}

Surprisingly, we also have the following special piece of bipotence.

\begin{lem}\label{1p1}\eroman Suppose $(\mcA,\mcA_0)$   is metatangible of the second kind.
   \begin{enumerate}
       \item  Then $a_1+a_2^\circ\in \{a_1, a_2^\circ\}$ for all $a_1,a_2\in \tT.$
    In particular,
    $\one + e$ is either $\one$ or~$e$.

      \item    $(-)(a_1+a_2^\circ) = (-)a_1+a_2^\circ$ for all $a_1,a_2\in \tT.$

    \item   $\one+ e =e $ if and only if $\one \preceq e.$

     \item $\mcA = \tT \cup \tT^\circ,$ and $\mcA_0=\tT^\circ$.
    \end{enumerate}
\end{lem}
 \begin{proof} (i)
     If $a_1+a_2^\circ\in \tT,$ then $  (a_1+a_2^\circ)(-)a_1\in \mcA_0,$ implying $a_1+a_2^\circ = a_1$ since $(\mcA,\mcA_0)$   is  uniquely negated. So assume $(a_1 + a_2) (-) a_2 \in \mcA_0.$ If $a_1 + a_2\in \tT$ then $a_1 + a_2 = a_2 $ since $(\mcA,\mcA_0)$ is uniquely negated, so $a_1+a_2^\circ = a_2 (-)a_2 = e.$ If $a_1+a_2 \in \mcA_0,$ then $a_1=(-)a_2,$ so $a_2 (-) (a_2+a_2) = a_1+a_2^\circ \in \mcA_0, $ implying $a_2+a_2 = a_2$ and again $a_1+a_2^\circ=a_2^\circ.$

     (ii) If $(a_1+a_2^\circ) =  a_2^\circ$ then $(-)(a_1+a_2^\circ) = (-)a_2^\circ = a_2^\circ.$
     If $(a_1+a_2^\circ) = a_1$ then $(-)a_1 +a_2^\circ = (-)(a_1+a_2^\circ) = (-)a_1.$

     (iii) If $\one \preceq e$ then $\one +e \preceq e+e \in \mcA_0,$ so $\one +e =e$ by (i). Conversely, if $\one +e =e$ then $(-)\one \preceq (-)\one + e = (-)\one + (\one +e) =e+e = e.$

     (iv) In view of (i), for the first assertion it suffices to prove $e+e\in \tT^\circ.$ But $e+e = (\one +\one)(-)(\one+\one),$ and $\one +\one \in \tT.$ Hence, any element $b\in \mcA_0$ cannot be in $\tT$ so is in $\tT^\circ,$ by Remark~\ref{TA}.
 \end{proof}

We also get a piece of strong $\tT$-reversibility for metatangible pairs.

 \begin{lem}\label{m2} \cite[Lemma~2.23]{Row26} Suppose  $(\mcA,\mcA_0,(-))$  is a metatangible triple. \begin{enumerate}\eroman
 \item Any surpassing relation is $\tT$-reversible.

   \item   If $\sum a_i \in \mcA_0$ for $a_i\in \tT$, then either $a_1=a_2=\dots =a_t$ or $(-)a_1 \preceq \sum_{i=2}^t a_i.$
   \end{enumerate}
 \end{lem}

\begin{lem}
       Suppose    $(\mcA,\mcA_0,(-),\preceq)$ is a metatangible system with   a  surpassing relation~$\preceq$. We write $mb$ for $\sum _{i=1}^m  b$, where $b\in \mcA$.
      \begin{enumerate}\eroman
          \item If $m\one \in \mcA_0$ with $m$ minimal and if $m$ is even, then $m=2$, i.e.,   $(\mcA,\mcA_0)$ is of the first kind. \item If either $(\mcA,\mcA_0)$ is of the second kind or $\one\preceq e,$ then $\preceq$ is strongly $\tT$-reversible.
      \end{enumerate}
\end{lem}
\begin{proof} (i) By minimality of $m,$ $a':= {\frac{m}{ 2}}\one \in \tT.$ Hence   $a'(\one+\one)= a'\one+a'\one  \in \mcA_0,$  implying $\one+\one  \in \mcA_0.$

(ii)
 Writing $b$ as a sum $\sum_{i=1}^t  a_i$ of elements of $\tT$, suppose $a +\sum_{i=1}^t a_i\in \mcA_0$. By Lemma~\ref{m2}(ii), we are done unless all $a_i=a,$ so we have $a(\one +\sum _{i=1}^t \one) =a +\sum _{i=1}^t a\in \mcA_0,$ so $\sum _{i=1}^{t+1}\one\in \mcA_0.$ We need to show $(-)\one \preceq t\one$.
  Take $m\le t+1$ minimal such that $m\one  \in \mcA_0$.

 First assume that $m$ is even. Then, by (i),   $(\mcA,\mcA_0)$ is of the first kind.  If $t$ is odd then $\one \preceq  \one +{\frac{t-1}2} e =  t \one,$ as desired. If $t$ is even then $\one \preceq  e +({\frac{t}2}-1) e =  t \one,$  as desired.

 So we have reduced to the case that   $m$ is odd $>2.$ Furthermore $\one + (m-1)\one \in \mcA_0,$ and  $(m-1)\one \in \tT$ by minimality of $m,$ implying $(m-1)\one = (-)\one.$ If $t$ is odd then
 $(-)\one \preceq  ( m-1)\one +(t-m+1) \one = t\one$.

Hence we may assume that $t$ is even. We use Lemma~\ref{1p1}. 
  If $\one + e = \one,$ then we can subtract $m$ from $t$ and conclude by induction. Hence $\one + e = e$, implying $(t-1)\one = e,$ and hence $\one \preceq \one  + (t-1)\one = t\one.$
\end{proof}

\subsection{Examples of pairs, triples, and systems}\label{expa}$ $

These are the pairs that will be relevant to this paper.

\subsubsection{Tropical extensions of    pairs}\label{ExP}$ $

Generalizing the supertropical semiring of \cite{IR}, recall from \cite[Definition~1.6]{Row26}
that given a $\tT$-module  $\mathcal L$ and  an ordered abelian semigroup   $(\tG,+)$,
the \textbf{tropical extension} $ \mathcal L \rtimes \tG$ of $\mathcal L$
is the set  $\mathcal L\times \tG$ endowed with the following addition:
   \begin{equation}
\label{basicex17}(\ell_1,g_1) + (\ell_2,g_2) = \begin{cases}
(\ell_1,g_1) \text{ if } g_1 > g_2,
 \\ (\ell_2,g_2) \text{ if } g_1 < g_2,  \\  (\ell_1 + \ell_2,\, g_1)
 \text{ if }   g_1 =  g_2  .
\end{cases}.\end{equation}


  If $\mathcal L$  is also a multiplicative monoid then $\mathcal L \rtimes \tG$ has componentwise multiplication $(\ell_1,g_1) (\ell_2,g_2) = (\ell_1\ell_2,g_1 + g_2)$.

    \begin{lem}\label{tr2}  If $(\tG,+)$ is an ordered abelian semigroup and   $(\mathcal L,\mathcal L_0)$ is a   pair, then $(  \mathcal L \rtimes \tG,  \mathcal L_0\rtimes \tG)$ is a     pair with underlying monoid $\tT\times \tG$,  under the action $$(a,g)(\ell,g') = (a\ell,g+g'),\qquad \forall a\in \tT,\ \ell \in \mathcal L,\ g,g'\in \tG.$$

    Any negation map $(-)$ on $(\mathcal L,\mathcal L_0)$ extends to $\mathcal L \rtimes \tG$ via $(-)(a,g)=((-)a,g).$
    If $(\mathcal L,\mathcal L_0)$ is metatangible, then so is  $(  \mathcal L \rtimes \tG,  \mathcal L_0\rtimes \tG).$

  If $\mathcal L$   is a  semiring then $\mathcal L \rtimes \tG$  is a  semiring.
    \end{lem}
    \begin{proof} The action is clear, and the negation map is given by
         $(-)(a,g)=((-)a,g).$ For $(a_1,g_1),(a_2,g_2)\in \tT\times \tG$, $(a_1,g_1)+(a_2,g_2)\in \{(a_1,g_1),(a_2,g_2)\}$    unless $g_2 = g_1,$  whereas $(a_1,g)+(a_2,g)=(a_1+a_2,g),$ so we also see that $(  \mcL \rtimes \tG,  \mathcal L_0\rtimes \tG)$ is   uniquely negated. Now the last two assertions  are clear.
    \end{proof}

 We extend a surpassing relation $\preceq$ on  $(\mathcal L,\mathcal L_0)$ to $(  \mathcal L \rtimes \tG,  \mathcal L_0\rtimes \tG)$, given by $(b_1,g_1) \preceq (b_2,g_2)$ if $g_1<g_2$ with $b_2\in \mathcal L_0$, or if $g_1=g_2$
 with $b_1\preceq b_2.$

\begin{definition}\label{supt}
     The \textbf{supertropical pair}, denoted  $(\mathbb {T}(\tG),\tG )$, is  $( \mathbb T\rtimes \tG,  \zero\rtimes \tG)$, where $\mathbb {T}$ is as in~Example~\ref{tr1}(ii), with $\tT =\{\one\}\times \tG$. Note that the supertropical pair is actually a triple, where the identity map is a negation map of the first kind.
\end{definition}

\subsubsection{Doubling}$ $

One way to create a triple is by \textbf{doubling} a $\tT$-module~$\mcA$, inspired by the familiar construction of $\Z$ from $\Net$:

\begin{definition}\label{doub}
    The \textbf{doubled}, or \textbf{symmetrized}, $\tT$-module   of a $\tT$-module~$\mcA$ is $\widehat{\mcA} = \mcA\times \mcA$, with  underlying monoid $\tT\! \times\! \{\zero\}\, \cup\,  \{\zero\}\! \times\! \tT $ and componentwise operations.
      Define   $\widehat{\mcA}_0 =\{(a,a): a\in \mcA\}$.

\end{definition}
  \begin{rem}$ $
   \begin{enumerate}\eroman
       \item $(\widehat{\mcA},\widehat{\mcA}_0)$ is a triple, where $(-)(\one,\zero) = (\zero,\one).$

          \item   When $\mcA$ is also  a  semiring, \textbf{twist multiplication} can be
defined on $\widehat{\mcA}$ by
$ (b_1,b_2)\ctw (b'_1,b'_2) =
 (b_1b'_1 + b_2 b'_2, b_1 b'_2 + b_2 b'_1)$ for $(b_1,b_2), (b'_1,b'_2) \in \widehat{\mathcal A}, $ making $(\widehat{\mcA},\widehat{\mcA}_0)$ a semiring triple.
   \end{enumerate}
 \end{rem}

\subsubsection{Hyperfields and hyperpairs}\label{hyp2}$ $

We follow \cite{Mit,krasner,Vir}.
Hyperfields formalize multivalued sums. Given a set $\mathcal H$,
define $\nsets(\mcH) := \mathcal{P}(\mcH) \setminus \{\emptyset\}.$

\begin{definition}\label{hy7} Recall that a \textbf{hyperfield} is a set  $\mathcal H$ with:
\begin{enumerate}
    \item a commutative multivalued addition $\boxplus : \mathcal{H}\times \mathcal{H}\to \nsets (\mathcal H),$ which is \textbf{associative }   in the sense that if we
 define
\[ a \boxplus S = S\boxplus a =\bigcup _{s \in S} \ a \boxplus s,
\]
 then $(a_1
\boxplus a_2) \boxplus a_3 = a_1 \boxplus (a_2\boxplus a_3)$ for all
$a_i$ in $\mathcal H .$

   \item   an  element $\Hzero$, satisfying $\Hzero \boxplus a = a \boxplus \Hzero = a$ for all $a\in \mathcal H$.

   \item  a unique \textbf{hypernegative} $-a \in \mathcal H  $, in the sense that  $\Hzero \in a \boxplus
   (-a),$ which also distributes over hyperaddition, in the sense that $-(a_1\boxplus a_2) = (-a_1)\boxplus(-a_2).$ Here $(-)S $ denotes $\{-a: a\in S\}.$

  \item  a multiplication $\cdot$, for which $(\mcH \setminus \{ \Hzero\},\cdot)$ is an abelian group, with $\Hzero$ absorbing, and  $\cdot$ distributing over hyperaddition.

 \end{enumerate}
  \end{definition}

\begin{lem}\label{hp22}
   Any hyperfield $(\mathcal{H},\boxplus ,\Hzero)$ induces
     a semigroup       $(\nsets(\mcH),\boxplus ,\{\Hzero\}),$ with addition  given by $$S_1 \boxplus   S_2 = \bigcup \{s_1 \boxplus  s_2: s_i\in S_i\},$$ and
     a monoid $(\nsets(\mcH),\cdot,\{\one\})$, with multiplication $S_1 \cdot S_2 = \{a_1a_2: a_i \in S_i\}$, which satisfies ``single distributivity'':

    \begin{equation}
         a \cdot \boxplus_i  S_i = \boxplus_i  (a\cdot S_i)  ; \quad (\boxplus_i  S_i)\cdot a =  \boxplus_i (S_i  \cdot a), \qquad \forall a\in \mcH, \ S_i \in \nsets(\mcH).
     \end{equation} Here $\mcH$ is viewed as the set of singletons in $\nsets (\mathcal H ),$ identifying $a\in \mathcal{H}$ with~$\{a\}$.
 $\nsets(\mcH)$ is an $\mathcal{H}$-module, with the action viewed elementwise.

 \end{lem}
 \begin{proof} Associativity and single distributivity are checked elementwise.
 \end{proof}

 \begin{definition} For a  hyperfield~$\mathcal H$,  let $\nsets(\mcH)_0= \{S\subseteq \mcH: \Hzero \in S\}$. The sub-pair $(\mcA,\mcA_0)$ of  $(\nsets(\mathcal H),\nsets(\mathcal H)_0)$ spanned by   ~$\mathcal H$ is called
     the \textbf{hyperpair} of~$\mathcal H$. $(\mcA,\mcA_0)$ is a triple with underlying group $\mcH \setminus \{ \Hzero\}$, where the hypernegative $(-)$ is   applied elementwise.  $(\mcA,\mcA_0,(-),\subseteq)$ is called
     the \textbf{hypersystem} of $\mathcal H$.
 \end{definition}

 Here,  the surpassing relation is set inclusion, since $\{a_1\}\subseteq\{a_2\}$ implies $a_1=a_2.$

Recall that a hyperfield $ \mcH$ is \textbf{stringent} if $a_1\boxplus a_2\in \mcH$ for all
   $a_1 \ne -a_2\in \mcH.$

\begin{rem}\label{warn} For a  hyperfield~$\mathcal H$,
\begin{enumerate}\eroman
    \item    when   $\mcH$ is stringent, the hyperpair of  $\mcH$ is metatangible.

\item the hypersystem of $\mathcal H$ is a semiring system when $\nsets(\mcH)$ is a semiring,\footnote{Warning: $\nsets(\mcH)$ is not a semiring when $\mcH$ is the phase hyperfield \cite[\S 7.4]{Vir}, cf.~Note~\ref{warn0}.} as in the following familiar examples of stringent hyperfields from \cite{Vir}:

\begin{itemize}\eroman

\item   Our motivation came from the \textbf{tropical hyperfield} \cite{Mit,Vir} which
  consists of  $\mcH:=\R \cup \{-\infty\}$, with $-\infty$
  as the zero element $\Hzero$ and $0$~as the unit element $\one$, equipped
  with  addition $a\boxplus b = \{a\}$ if $a>b$,
  $a\boxplus b = \{b\}$ if $a<b$,
  and $a\boxplus a= [-\infty,a]$. It is easy to see that there is an embedding from the supertropical pair to the tropical hyperfield, given by $a \mapsto a$ and $a ^\circ\mapsto [-\infty,a].$

  A special case is  the \textbf{Krasner hyperfield} \cite{krasner}  $ \mathcal K := \{ 0, 1 \}$, with the usual multiplication
law, and with hyperaddition defined by  $ x \boxplus  0 = 0 \boxplus  x =
x $ for all $x ,$ and $1 \boxplus  1 =  \{ 0, 1\}$.   $\nsets(\mathcal K) = \mathcal K\cup \{\mathcal K\}$ is isomorphic to~$\mathcal{B}$, by the map $0 \mapsto \zero, \ 1\mapsto \one, \ \{0,1\}\mapsto  \one^\circ.$

\item
  The \textbf{hyperfield of signs}
$ \mathcal S := \{ 1 , 0, -1\}$,
  with the usual multiplication
law, and hyperaddition defined by $1 \boxplus  1 = 1 ,$\ $-1
\boxplus  -1 = -1  ,$\ $ x \boxplus  0 = 0 \boxplus  x = x $ for all
$x ,$ and $1 \boxplus  -1 = -1 \boxplus  1 = \{ 0, 1,-1\} $.
$\mathcal{S}\cup  \{\mathcal S\}$ is a semiring isomorphic to $\widehat{\mathbb T},$ by sending $0 \mapsto \zero,$ $1 \mapsto (1,0)$,  $ -1 \mapsto (0,1)$, and $\{ 0, 1,-1\} \mapsto (1,1).$
  \end{itemize}
\end{enumerate}
\end{rem}

Hypersystems are quite well behaved, as seen in the following results taken from \cite[Lemmas 2.29 and 2.30]{Row26}.

\begin{itemize}
    \item  The surpassing relation $\subseteq$ on a hypersystem satisfies fissure, and thus
     is strongly  $\tT$-reversible by  Lemma~\ref{fis}.

    \item    Every  hypersystem   of a hyperfield  is a strong domain.
\end{itemize}

\subsubsection{The monoid pair}$ $
\begin{definition}
    As in \cite[Definition~1.21]{Row26}, for any $\tT$-module $\mcA$ and monoid $\mcM,$ we define $\mcA[\mcM]$ to consist of formal sums $\sum_{u\in \mcM} b_u u :  \, b_u \in \mcA $,  under componentwise addition $ \sum b_u u + \sum b'_u u =\sum (b_u+b'_u)u$,  which  is a module over the monoid $\tT_\mcM = \{ a u : a\in \tT,\ u \in \mcM\},$
    under the action $(au)\sum b_v v = \sum(a b_v)uv.$
 The \textbf{monoid pair} $\mcA[\mcM]$ of a monoid $\mcM$ over a pair $(\mcA,\mcA_0)$   is $(\mcA[\mcM],\mcA_0[\mcM]).$
\end{definition}

 \subsection{Localization} $ $

 Later we shall want to work with pairs over groups, so we need a localization lemma. Any cancellative monoid has its Grothendieck group of fractions, in which $\frac{a_1}{a_2}$ is represented by the pair $(a_1,a_2)$ under the identification $(a_1,a_2)\equiv (a_1',a_2')$ when $a_1a_2' = a_1'a_2.$

 \begin{lem}\label{loc} Suppose  $\tG$ is the Grothendieck group of fractions of the monoid $\tT$. Then
any pair $(\mcA,\mcA_0)$  with underlying monoid $\tT,$ has a quotient pair $(\mcA_\tG,{\mcA_0}_\tG)$ where $\mcA_\tG = \{ \frac ba : b\in \mcA, a\in \tT\}$ and ${\mcA_0}_\tG = \{ \frac ba: b\in \mcA_0\} $ under the identification $\frac{b_1}{ a_1} = \frac {b_2}{a_2}$ if $a_2 b_1 = a_1 b_2.$ If $(\mcA,\mcA_0,(-))$ is a triple, then $({\mcA}_\tG,{\mcA_0}_\tG,(-))$ is a triple with $(-)\frac{b}{a}= \frac{(-)b}{a} .$ Any surpassing relation $\preceq$ lifts, by defining $\frac{b_1}{ a_1}\preceq \frac {b_2}{a_2}$ when $a_2b_1 \preceq a_1 b_2.$ Being $\tT$-reversible (resp. strongly $\tT$-reversible) lifts to $\tG$-reversible (resp. strongly $\tG$-reversible), as does the property of $(-)$-domain.\end{lem}
\begin{proof}
    The verifications are routine;   furthermore, this construction is a special case of \cite[Theorem~4.5]{JMR3}. For example, if $(\mcA,\mcA_0,(-))$ is a $(-)$-domain and $\frac{a_1}{a_1'}\frac{h}{a} (-) \frac{a_2}{a_2'}\frac{h}{a}= \frac{b}{a'}\in {\mcA_0}_\tG$, then $a_1a_2'h (-) a_1'a_2 h \in \mcA_0,$  implying either $h\in \mcA_0$ or $a_1a_2' = a_1'a_2$, yielding $\frac{a_1}{a_1'}= \frac{a_2}{a_2'}$.
\end{proof}

\section{Polynomial pairs and their roots} $ $

For any $\tT$-module $\mcA,$
  $\mcA[\lambda]$ is the module over the underlying monoid $\tT_\la$ of monomials in the commuting indeterminate $\lambda$ with coefficients in $\tT$.
  The elements of   $\mcA[\lambda]$ are called \textbf{polynomials}. A polynomial which is the sum $a_1\la^i +a_2 \la ^j$  of two monomials from $\tT_\la$ is a \textbf{binomial}. We shall see that the binomials play a special role.

 For a polynomial $g:= g(\la) = \sum_j b_j\la^j$ and  $a\in \tT,$    define $g(a) = \sum b_j a^j.$\footnote{Note that $b_ia^i$, and thus $g(a),$ are defined since $a^i \in \tT$.}

  \begin{itemize}
      \item   Write $f\cong g$ if  $f(a)=g(a)$ for all $a\in \tT.$
  \item     Write $f\equiv g$ if  $f(a)=g(a)$ for almost all $a\in \tT.$ (Of course this is tautological unless $\tT$ is infinite.)

  \end{itemize}

\begin{definition}\label{null1} Suppose $(\mcA,\mcA_0)$ is a pair with underlying monoid $\tT$.

\begin{enumerate}
\item A polynomial $g$ is \textbf{null} if $g(a)\in \mcA_0$ for  all $a\in \tT.$

 \item A polynomial $g$ is \textbf{almost null} if $g(a)\in \mcA_0$ for almost all $a\in \tT.$

 $\mcA[\la]_0$~is the set of almost null
   polynomials.
 \item  A nonzero polynomial is \textbf{tangible} if its coefficients are all in $\tTz.$

 Throughout the sequel, $f = \sum_{i=0}^n a_i\la^i$  is always a tangible polynomial of degree~$n.$
    \item  The familiar  \textbf{convolution product} of a tangible polynomial  $f$  with an arbitrary polynomial $g = \sum b_j\la^j$ is given by \begin{equation}
     \label{conp} \left(\sum a_i\la^i\right)\left(\sum b_j\la^j\right) = \sum_k\sum_{i+j=k}   (a_i b_j)\la^k, \qquad a_i\in \tT,\ b_j\in\mcA.
\footnote{In this exposition we do not require $\mcA$ to be endowed with multiplication. When $\mcA$ is a semiring we can   define the convolution product of two arbitrary polynomials, taking  $a_i,b_j\in \mcA.$}  \end{equation}

\end{enumerate}
\end{definition}

\begin{example} \label{basic1}
    In \eqref{conp}, we have a bit of distributivity. Namely, the following equality is quite important: \begin{equation}
        \label{com3}
        (\la (-)a)g = (-)ab_0 + \sum_{i=1}^n (b_{i-1}(-)a b_{i})\la^i +b_n \la^{n+1} = \la g (-) a g.
    \end{equation}

    In the special case that $(\mcA,\mcA_0)$ is  metatangible, this is a tangible polynomial unless $ b_{i-1}(-)ab_i \in \mcA_0$ for some $i,$ i.e., $b_{i-1} = ab_i .$
\end{example}

In studying polynomials over a pair $(\mcA, \mcA_0)$ (continuing the running hypotheses of Note~\ref{mulr}), the naive choice of a pair would be $(\mcA[\la], \mcA_0[\la]),$
 following \cite[Example~1.22 and~Example~2.28(vi)(a)]{Row26}. However, we usually prefer to take  $\mcA[\lambda]_0$ as defined above.
  Note that $\mcA_0[\lambda]\subseteq \mcA[\lambda]_0$, but equality need not hold.

  \begin{definition}
      The \textbf{abstract polynomial pair} is
$(\mcA[\la],\mcA _0[\la])$, whereas   the \textbf{polynomial function pair} is
$(\mcA[\la],\mcA [\la]_0)$, both pairs taken with underlying monoid $\tT_\la :=\bigcup_{i\ge 0} \tT \la^i,$ i.e., the tangible monomials. (For abstract polynomials, we view~$a\la^i$ as an abstract monomial, whereas for polynomial functions, we view  $a\la^i$ as a monomial function.)
  \end{definition}

 \begin{example}\eroman
     $ $
     \begin{enumerate}
      \item For a finite field $F$ of $p$ elements,  and $(-)$ the usual negation, $F_0 = 0,$ so $F_0[\la]=0,$ but the polynomial $\la^p (-)\la $ is null.

         \item Over the hyperpair of a hyperfield~$\mcH$, the tangible polynomials have coefficients in $\mcH$.
     \end{enumerate}
 \end{example}

 For the remainder of this paper, $\tT$ is an infinite torsion-free monoid, in the sense that for all $a\ne \one$ in $\tT$, $a^n\ne \one$ for all $n> 1;$ thus $a^i\ne a^j$ for all $i\ne j.$ Define the pre-surpassing relation $\prect$ on $(\mcA[\la] ,\mcA[\la]_0 )$ elementwise, by $g\prect h$ if and only if $g(a)\preceq h(a)$ for almost all $a\in \tT.$

 \begin{lem}\label{hyp1} Suppose $\mcA$ is a $\tT$-module.
     \begin{enumerate}\eroman
     \item If $a_1a^i = a_2 a^j$ and $a_1a^{ki} = a_2 a^{kj}$ for some $k>1,$ then $i=j.$
         \item If $a_1 \lambda^i \prect a_2 \lambda^j$ then $a_1=a_2$ and $i=j$.
     \end{enumerate}
 \end{lem}
 \begin{proof}
 (i) $a_1a^ia^{(k-1)i} =a_1a^{ki} = a_2 a^{kj}= a_2a^ja^{(k-1)j}= a_1a^ia^{(k-1)j}, $ implying $a^{(k-1)i}=a^{(k-1)j}.$ Hence  $(k-1)i =(k-1)j,$ implying $i=j.$

    (ii)  Fix $a\in \tT$. For almost all $b\in \tT,$ we have $a_1 (ab)^i=a_2 (ab)^j$
      and  $a_1 b^i=a_2 b^j;$  hence $a^i = a^j$ for each $a\in \tT.$ Thus $i=j.$  But then also $a_1 = a_2.$
 \end{proof}

\begin{lem}\label{pol1} Let $(\mcA ,\mcA _0,(-))$ be a pre-triple with a pre-surpassing relation $\preceq$.
    \begin{enumerate}\eroman
     \item Both $(\mcA[\la] ,\mcA_0[\la])$ and $(\mcA[\la] ,\mcA[\la]_0)$ have a negation map, obtained by applying $(-)$ to each   coefficient.
        \item Define the pre-surpassing relation $\preceq$ on $(\mcA[\la] ,\mcA_0[\la],(-) )$ coefficient-wise, i.e., $\sum b_i \la^i \preceq \sum b'_i\la^i $ if and only if $b_i\preceq b_i'$ for all $i.$ Then $\preceq$ is
        \begin{enumerate}
            \item  $\tT_\la$-reversible if $\preceq$  is  $\tT$-reversible on $(\mcA ,\mcA _0)$.

        \item  strongly $\tT_\la$-reversible if $\preceq$ is strongly $\tT$-reversible on $(\mcA ,\mcA _0)$ .
        \end{enumerate}
         \item  The pre-surpassing relation $\prect$  on  $(\mcA[\la] ,\mcA[\la]_0 ,(-) )$ is
        \begin{enumerate}
\item    $\tT_\la$-reversible if $\preceq$ is  $\tT$-reversible on  $(\mcA ,\mcA _0)$.

\item  strongly $\tT_\la$-reversible if $\preceq$ is strongly $\tT$-reversible on  $(\mcA ,\mcA _0)$.
        \end{enumerate}
    \end{enumerate}
\end{lem}
\begin{proof} (i) $ \sum b_i\la^i +\sum (-)b_i\la^i = \sum  b_i^\circ \la^i \in \mcA_0[\la] \subseteq \mcA[\la]_0.$

    (ii)  For reversibility, the main concern is about monomials.

    (a)  If  $(-)a_0 \la^{j_0} \preceq \sum_{j}\sum_{i=1}^n a_{i,j}\la^{j}$  then  $\zero \preceq  \sum_{i=1}^n a_{i,j}  $ for all $j\ne j_0$ and $(-) a_0 \preceq \sum_{i=1}^n a_{i,j_0}, $  implying $(-) a_{1,j_0} \preceq  a_0 +\sum_{i=2}^n a_{i,j_0}$,  and thus  $(-) a_{1,j_0}\la^{j_0} \preceq \left(a_0+\sum_{i=2}^n a_{i,j_0}\right)\lambda^{j_0}+\sum_{j\ne j_0}\sum_{i=1}^n a_{i,j}\lambda^j.\)

    (b)  If  $a \la^i +b\in \mcA_0$ for $a \in\tT$, and $b = \sum b_j \la^j$, then  $ {a }+ b_i\in \mcA_0$  so
    $ (-)a  \preceq b_i,$ and $b_j\in \mcA_0$ for all $j\ne i$, implying $(-) a \la^i \preceq b$.

    (iii) To see that $\prect$ is a pre-surpassing relation, we first note that if $h(\la)\in \mcA[\la]_0,$ then $h(a)\in \mcA_0$ for almost all $a\in \tT$, in which case $g(a)\preceq g(a)+h(a) = (g+h)(a),$ and thus $g(\la) \prect  (g+h)(\la).$

    Furthermore, we need to show that if $a_1 \lambda^i \prect a_2 \lambda^j$ then $a_1 \lambda^i=  a_2 \lambda^j$ as
    functions. But $a_1 a^i = a_2 a^j$ for almost all $a\in \tT,$ so we are done by Lemma~\ref{hyp1}.

    (a)   If  $(-) a_0\la^{i_0} \prect \sum_{j=1}^n a_j\la^{i_j}$  then $(-) a_0a^{i_0} \preceq \sum_{j=1}^n a_j a^{i_j}$ for  almost all $a\in \tT$, implying $(-) a_1a^{i_1} \preceq   a_0a^{i_0}+ \sum_{j=2}^n a_j a^{i_j},$  so $(-) a_1\la^{i_1} \prect   a_0\la^{i_0}+ \sum_{j=2}^n a_j \la^{i_j}.$

     (b)  If  $a_1 \la^i +g(\la)\in \mcA[\la]_0$ for $a_1 \in\tT$ then, for almost all $a\in \tT,$ $ { a_1 a^i }+ g(a) \in \mcA_0$,   so $(-) a_1 a^i \preceq g(a),$ implying $(-) a_1\la^i \prect g(\la)$.
\end{proof}

A major difference   concerns the $(-)$-domain property. Let $b = e ^2 +\one.$ If $b\in \mcA_0,$ which  is the case  in many non-classical examples, then
$(\mcA[\la] ,\mcA_0[\la],(-) )$  is not a strong domain, since $$(e \la +\one)(\la + e) = e \la^2 +b \la +e.$$  The abstract polynomial hypersystem  over   a hyperfield
 is not even a $(-)$-domain  since \begin{equation}
    \begin{aligned}(\la (-) \one)&(\{\zero, \one\}\la^2 +\la +\{\zero,\one\}) = \\  &\{\zero, \one\}\la^3+\{\one, \one(-)\one\}\la^2 +\{ (-)\one, \one (-)\one\} \la +\{\zero, (-)\one\}\in  \nsets(\mcH)_0[\la].
    \end{aligned}
\end{equation}
The story is different for polynomial functions.

  \begin{lem}\label{domup}
        If a pre-triple  $(\mcA ,\mcA_0,(-))$ is a  $(-)$-domain, then $(\mcA[\la] ,\mcA[\la] _0,(-))$  is a $(-)$-domain.
  \end{lem}
  \begin{proof}
 Suppose $a_1 \la^i \ne a_2 \la^j$ and $a_1 \la^i g(-) a_2 \la^j g \in\mcA[\la]_0$. Then, for almost all $a\in \tT,$ $a_1 a^i \ne a_2 a^j$ and $a_1 a^i g (a)(-) a_2 a^jg (a) \in\mcA_0$, implying $g (a) \in\mcA_0$. Hence  $g  \in\mcA[\la]_0.$
  \end{proof}

Although a $(-)$-domain, $(\mcA[\la] ,\mcA [\la]_0,(-))$ need not be uniquely negated, since if $f_1\equiv f_2$ are tangible polynomials, then $ f_1  (-)f_2\in \mcA[\la]_0.$
Accordingly, we say that $(\mcA[\la] ,\mcA [\la]_0,(-))$ is \textbf{essentially uniquely negated} if $f \equiv g$ for tangible polynomials $f$ and $g$ implies $(-)f\equiv (-)g.$ This is all we need in many proofs, since then $f^\circ \equiv g^\circ.$

\begin{lem}\label
 {uniqn}
    Suppose $(\mcA ,\mcA _0,(-))$ is uniquely negated and metatangible.
    \begin{enumerate}
        \item $(\mcA[\la] ,\mcA _0[\la],(-))$ is uniquely negated.

        \item $(\mcA[\la] ,\mcA [\la]_0,(-))$ is essentially uniquely negated.
    \end{enumerate}
\end{lem}
\begin{proof}
    (i) Compare coefficients.

    (ii) $ f(a)= g(a)$ implies $(-) f(a)= (-)g(a)$.
\end{proof}

We can improve Lemma~\ref{sur1}.

\begin{lem}\label{sur1a} $\prect$ is a surpassing relation if $(\mcA,\mcA_0)$ is uniquely negated.\end{lem}
\begin{proof}
If $f_1 +g = f_2$ for $f_i\in \tT_\la$ and $g\in \mcA_0[\la],$ then $f_2(-)f_1 = f_1^\circ +g \in \mcA[\la]_0$. For any $a$ such that $f_1(a),$ $f_2(a)$ are tangible and $g(a)\in \mcA_0$, we see that
  $f_2(a) = f_1(a)$ since $(\mcA ,\mcA_0 )$ is uniquely negated,
   implying $f_1(a)+g(a) \preceq f_2(a)$ and thus $f_1\prect f_2.$
\end{proof}

  Fissure does not seem to lift.

 \subsection{Null roots and $\preceq$-roots of polynomials}$ $

   In classical algebra, $a$ is a root of~$f$  when $f(a) = \zero$.
But this must be modified to make sense for our main examples, in which $f(a) = \zero$  cannot hold for $f\ne \zero$.\footnote{This often is circumvented  in the literature by defining a root of a pair $(f,g)$ of polynomials over a semiring to be some $b\in \mcA$ such that $f(b)=g(b).$ This is the same as saying that $b$ is a root of $(f,g)$ in the doubled polynomial pair, in our sense.}
Accordingly, we have a different notion of root.

\begin{definition}\label{subs}  For any tangible polynomial $f$, \begin{itemize}
    \item  An element $a\in \tTz$ is a \textbf{null root} of  $f$  if $f(a)\in \mcA _0,$ i.e., $a \in f^{-1}(\mcA_0)$.

     \item   An element $a\in \tTz$ is a $\preceq$-\textbf{root} of  $f = \sum _{i=0}^n a_i\la^i$, if $(-)a_n a^n \preceq \sum_{i=0}^{n-1} a_ia^i$.
\end{itemize}
\end{definition}

By definition, null roots are  taken only from $\tTz.$

\begin{rem}\label{pre1}
   When our surpassing relation satisfies the property:
  \begin{equation}
      \label{pre2}  \zero \preceq b\Rightarrow b\in \mcA_0,
  \end{equation}
as with $\preceq_0$ and hyperpairs, then every $\preceq$-root is a null root; the converse requires~$\preceq$ to be strongly  $\tT$-reversible.
\end{rem}

  We want to see how far a monic polynomial is determined by its null roots, as in classical algebra over an algebraically closed field.

\begin{example}\label{nonuniq}
  $ $\begin{enumerate}\eroman
    \item We present infinitely many tangible polynomials which have the same null roots.  Define  $f_\alpha = \la^2 +\alpha \la +\one$ over the supertropical pair $(\mathbb {T}(\tG),\tG )$ of Definition~\ref{supt},  for $\alpha < \one$ in $\tT$.
         \item    Suppose $f = \sum _{i=0}^n a_i\la^i$, and write $f_1= \sum _{i=1}^n a_i\la^i $, so $f = f_1+ a_0$. Then
    $\zero$ is a null root of $f$ if and only if $a_0\in \mcA_0,$ which for $f$ tangible means $a_0 = \zero$ and $f=\la f_1.$

    \item   Without some extra assumption, when $\mcA$ is idempotent, the polynomial $f = \la^2 (-) \la + \one$ has $\one$ as a null root but not a $\preceq$-root. More generally,
    suppose $\tT$ is a group, and $f = \la^2 (-) a_1\la + a_0,$ and we want to write $f \preceq (\la (-) a)g$ for $g = \la (-)b,$ for $b \in \tT.$ Then  $a_0 \preceq (-)a b ,$ implying    $a_0 = (-)a b .$   If $a=\zero$ then $a_0=\zero$ and we can take $g = \la (-) a_1.$ If $a\ne \zero,$ then
     $ b=  a_0 a^{-1} ,$ so $a_1 \preceq   a + a_0 a^{-1},$ implying $ a_1a \preceq   a^2 + a_0,$ so we need $a^2 (-) a_1 a + a_0 \in \mcA_0$ to imply $a_1a \preceq   a^2 + a_0$, which would hold for a null root of $f$ in the presence of strong $\preceq$-reversibility.
 \end{enumerate}
\end{example}

\begin{definition}
    \label{d1}
For a polynomial $g:= g(\la) = \sum_j b_j\la^j$, define $\supp g = \{j: b_j\ne \zero\}$.
   A  \textbf{sub-polynomial} $h$ of a polynomial $f$ is a sum of some of its monomials. Clearly $\supp h \subseteq \supp f,$ and any sub-polynomial of a
tangible   polynomial is tangible.
 Given two polynomials $f$ and $g,$ we define  $(f;g)$ to be the common sub-polynomial  of $f$ and $g.$ In other words, $f = h +\bar f$ and
 $g =  h +\bar g,$ where $h =(f;g)$, and $\supp h$ is disjoint from both  $\supp \bar f$ and  $\supp \bar g$.
\end{definition}

For instance, in   Example~\ref{nonuniq}(i), $(f_{\alpha_1};f_{\alpha_2})=\la^2 + \one$.
In fact, calling  $f$  a \textbf{minimal polynomial}  for a null root $a\in \tT$ if $a$ is not a null root of any proper sub-polynomial of $f,$ we have:
\begin{lem}
     Over  a metatangible triple  $(\mcA ,\mcA_0,(-))$,  any minimal polynomial of smallest degree for a null root $a$ is unique, up to multiplication by  elements of $\tT$.
\end{lem}
\begin{proof}
    Suppose $f=\sum _{i=0}^n a_i \la^i$  is a minimal polynomial for $a$, with $a_i\in \tT,$ and $g = \sum _{i=0}^{n} a_i' \la^i$  also is a minimal polynomial for $a$. Multiplying $f$ by $a_n'$ and $g$ by~$a_n,$ we may assume that $a_n = a_n'$. Write $f = h +\bar f$ and
 $g =  h +\bar g,$ where $h =(f;g)$. If $f\ne g$ then clearly $h\ne f$ and $h\ne g,$ so by hypothesis $h(a)$ is tangible, and furthermore $\la^n \in \supp h,$ so $\bar f(a)$ and $\bar g(a)$ are tangible. But $\bar f(a) = (-)h(a) = \bar g(a)$ since $(\mcA ,\mcA_0)$ is uniquely negated,  so $a$ is a null root of the  polynomial $\bar f (-)\bar g,$ which is tangible since the supports of  $\bar f $ and $\bar g$ are disjoint. $\bar f (-)\bar g$ would contain a minimal polynomial of degree $<n,$ a contradiction.
\end{proof}

\subsection{Factorization  of polynomials with respect to factor-roots}$ $

In classical algebra, an element $a$ is a root of $f$ if and only if $\lambda -a$ divides $f.$ Thus, we want an appropriate notion of  divisibility. The straightforward definition is not useful in general for polynomials and matrices, so here factorization always is carried out with respect to a given surpassing relation $\preceq$, into tangible polynomials.

\begin{rem}[{\cite[Lemma~3.8]{Row26}}]\label{div}
    If $f\preceq h_1\dots h_t$ for tangible polynomials $h_i,$ then $\sum _{i=1}^t \deg h_i = n,$ since the leading monomials match.
\end{rem}

 \begin{definition}\label{proot} $ $ \begin{enumerate}
     \item We write $f_1 \, |_\preceq \, f_2$ in $(\mcA[\la],\mcA _0[\la])$ if  $f_2 \preceq g f_1 $  for some   tangible polynomial~
 $g $.


 \item  An element $a\in \tTz$ is a  \textbf{factor-root} of $f$ if $(\la(-) a) \, |_\preceq \, f$.

  \item The polynomial $f$ is  \textbf{root-irreducible} if it has no null roots.
 \end{enumerate}
\end{definition}
(Were $g$   not required to be tangible in (1), the definition often would become vacuous.)
If $(\la (-) a) \ |_\preceq \ f,$ then $\deg g = n -1,$ by Remark~\ref{div}.


 \begin{lem}[{\cite[Lemma~3.9]{Row26}}]\label{rt0} Given $a\in \tT,$ let $g_{a,k} = \sum _{j=0}^{k-1}a^{k-j-1}\la^j. $
 \begin{enumerate}\eroman
 \item  $(\lambda (-) a)g_{a,k} = \la^k (-)  a^k  + \sum _{j=1}^{k-1}  (a^j(-) a^j)\la ^{k-j} $.

 \item $\la^n (-) a^n  \preceq (\la(-) a)g_{a,n}$.

  \item Suppose $f = \sum_{i=0}^n a_i \la^i \in \mcA[\la],$ and $g=  \sum _{k=1}^{n} a_k g_{a,k} \in \mcA[\la].$ Then $$(\lambda (-) a)g = f(\la) (-) f(a) + \sum_{i=0}^n\sum_{k=1}^{i-1}   a_i  (a^k(-) (a^k) )\la ^{i-k} .$$
  Hence $   f(\la)  (-) f(a)  \preceq (\lambda (-) a)g $.
\end{enumerate}
\end{lem}

  But $g$ need not be tangible. For the most robust theory, we need:\medskip

   \textbf{Hypothesis~R}: Every null root of $f$   is a  factor-root  of $f$.

\begin{theorem} [{\cite[Proposition 3.10, Theorem~3.13 and Corollary 3.14]{Row26}  for abstract polynomials}]\label{Rev}   Suppose the surpassing relation of the system   $(\mcA,\mcA_0,(-),\preceq)$ is $\tT$-reversible.   \begin{enumerate}\eroman
    \item Every  factor-root of $f$   is a null root of $f$.

    \item    Conversely, if moreover $\preceq$  satisfies fissure,  then   Hypothesis~R holds.
\end{enumerate}
\end{theorem}

\begin{example}\label{m1} In view of Example~\ref{basic1}, when  $g =\sum _{i=0}^{n-1}b_i \la^i$ is tangible, $f \preceq (\la (-)a)g$ means $a_0 = (-)ab_0$, $a_n = b_{n-1}$, and $a_i \preceq b_{i-1}(-)ab_i$ for all $1 \le i< n.$
In particular, when   $\mcA$ is idempotent, the polynomial $f =\la^2 + \la (-) \one$ has  the null root $\one,$ but $f \preceq ( \la (-) \one )( \la + \one)$
if and only if $\one \preceq \one (-)\one= e,$ which may fail.\end{example}
  \begin{ques}\label{eone}
      Does Hypothesis~R hold over pairs of tropical type, as defined in \cite[Definition~4.6]{AGR2}?
  \end{ques}


\subsection{Simultaneous factor-roots of a polynomial}$ $

\begin{MNote}\label{nonass} We turn to the question of factoring polynomials with respect to several roots.  An inconvenient fact is that the product of two tangible polynomials need not be tangible.\footnote{Appendix~A is the first step towards remedying this situation.}
    This  complicates repeated products of polynomials, and we have only defined products of two polynomials when one of them  is tangible; furthermore, even when $\mcA$ has associative  multiplication, multiplication in $\mcA[\la]$ need not be associative. For example, \begin{equation}\label{nona}
    \begin{aligned}
(\la (-)a_1)&( (\la (-)a_2)  h)  = (\la (-)a_1) (\la h (-) a_2h) \\ & = \la (\la h (-) a_2h) (-)a_1 (\la h (-) a_2h)= \la^2 h (-) \la (a_1 h + a_2 h) +a_1a_2 h.
     \end{aligned}
 \end{equation}

On the other hand, $  ((\la (-)a_1) (\la (-)a_2) ) h = (\la^2 (-)(a_1+a_2)\la +a_1a_2)h$, which is the same as \eqref{nona} when $h$ is tangible; however, for example, taking $$h = ((\la (-)a_3) (\la (-)a_4) )= \la^2 (-)(a_3+a_4)\la +a_3a_4,$$  the coefficient of $\la^2$ in  $ ((\la (-)a_1) (\la (-)a_2) ) h $  contains $(a_1+a_2)(a_3+a_4)\la^2$, and need not match  the coefficient of $\la^2$ in~\eqref{nona}.

     Accordingly, when multiplication in $\mcA[\la]$ is not required to be  associative, we adopt the convention that the product $f_1\dots f_m$ means $f_1(f_2(f_3\dots f_m)).$ When   $f_1, \dots, f_{m-1}$ are tangible, this is defined in view of \eqref{conp}, and for $f_i = \la(-)a_i$ we have a straightforward formula obtained by iterating \eqref{nona}.
\end{MNote}
 \begin{proposition}\label{sp1} We can write  $f \preceq \prod_{i=1}^t (\la (-)a_i )g$  for suitable $a_i\in \tT$ (possibly with repetitions) and a tangible factor-root-irreducible polynomial $g$.
  \end{proposition}
  \begin{proof}
      Write $f \preceq (\la (-) a_1 )g_1$ for $g_1$ tangible, and iterate until the right-hand factor has no factor-roots.
  \end{proof}

  Perhaps surprisingly, the order of the $a_i$ is irrelevant.

 \begin{lem}\label{nonb}
     For any $k_i\in \Net,$ $a_1,\dots, a_m\in \tT$, any permutation $\pi \in S_m$, and $h\in \mcA[\la] , $  $  \prod _{i=1}^m(\la^{k_{\pi (i)}}(-)a_{\pi (i)})h = \prod_{i=1}^m (\la^{k_i}(-)a_i)h.$
 \end{lem}
 \begin{proof}
    Case I: $m=2.$  $(\la^{k_1} (-)a_1) (\la ^{k_2}(-)a_2)  h = (\la^{k_2} (-)a_2) (\la ^{k_1}(-)a_1)h,$  seen by rearranging $a_1$ and $a_2$ in the following generalization of \eqref{nona}:
    \begin{equation}
        \label{nona1}\begin{aligned}
(\la ^{k_1}(-)a_1)&( (\la^{k_2} (-)a_2)  h)  = (\la^{k_1} (-)a_1) (\la^{k_2} h (-) a_2h) \\ & = \la ^{k_1}(\la ^{k_2} h (-) a_2h) (-)a_1 (\la^{k_2} h (-) a_2h)\\& = \la^{k_1+k_2} h (-)  a_2 h \la^{k_1}(-)a_1 h\la ^{k_2} +a_1a_2 h.
     \end{aligned}
    \end{equation}
    The right side is invariant under switching the indices 1 and 2, proving the assertion.

    Case II: general $m$.  Suppose $\pi(t) = m$. If $t = m,$ we are done by induction, using $(\la ^{k_m}(-)a_m)h$ instead of $h$.  If $t<m$, by Case I, taking $\sigma = (t \ t+1)\pi$,  $  \prod (\la^{k_{\pi (i)}}(-)a_{\pi (i)})h = \prod (\la^{k_{\sigma (i)}}(-)a_{\sg (i)})h,$ and we conclude by reverse induction on~$t$. (Applying $(t+1 \ t+2)$ etc. we can move $\la^{k_m} (-)a_m$ to the last position, and then are done as above.)
 \end{proof}

\begin{definition} $ $
   \begin{enumerate}
     \item    A \textbf{partial $\preceq$-splitting} of~$f$ \textbf{with respect to $h$} is an expression of the form
      $fh \preceq \prod_{i=1}^t (\la^{k_i} (-)a_i )g$

       \item   A partial $\preceq$-splitting of~$f$  is called a \textbf{$\preceq$-splitting} if $g = \one,$ i.e., if one can write  $fh\preceq \prod_{i=1}^n (\la^{k_i} (-) a_i)$ for suitable $a_i\in \tT$.

       \item    The tangible polynomial $f$ \textbf{$\preceq$-splits}  if $f$ has a $\preceq$-splitting (with respect to some $h$).
   \end{enumerate}
\end{definition}

Is a $\preceq$-splitting  with respect to $h=\one$  unique?

\begin{example}\label{ce} (cf.~\cite[Example~1.9]{BaL})
  The \textbf{hyperfield of weak signs} is
$ \mathcal S_w := \{ 1 , 0, -1\}$ with the usual multiplication
law, but with hyperaddition now defined by $1 \boxplus  1 =  -1
\boxplus  -1 = \{-1,1\}  ,$\ $ x \boxplus  0 = 0 \boxplus  x = x $ for all
$x ,$ and $1 \boxplus  -1 = -1 \boxplus  1 = \{ 0, 1,-1\} $.
 The additive submonoid of $\nsets(\mathcal S_w)$ generated by $\mathcal{S}_w$ is
$$\mathcal Q:=\{ \{0\},  \{1\},  \{-1\},  \{-1,1\},  \mathcal{S}_w\}$$ in which the polynomial $f:=\la^2 +\la+1$ has two factor-roots $1,-1.$ 
Note that $(\la +1)^2 = \la^2 +\{-1,+1\}\la +1 = (\la (-)1)^2,$ so $f\preceq (\la +1)^2 $ but also $f\preceq (\la (-)1)^2 .$ Thus the $\preceq$-splitting (with respect to $h=\one$) is not unique. But     $\mathcal{Q} $
 fails to be a semiring, since
$(1\boxplus 1)^2 = \{-1,1\}$ whereas $(1\boxplus 1)\boxplus (1\boxplus 1)=  \{-1,1\} \boxplus \{-1,1\} =\mathcal{S}_w.$  \end{example}

On the other hand, we have

\begin{lem}\label{indst0}
  Suppose $(\mcA,\mcA_0,(-))$ is a $(-)$-domain  satisfying  Hypothesis~R.   Then the following holds:
  \begin{itemize}
      \item  \textbf{Root Condition}. If  $f \preceq (\la (-) a_1 )g$ for $a_1\in \tT$ and a tangible polynomial~$g$, and
      $a\ne a_1$ is a  null root of $f$, then   $a$ is a  null root of $g$.
  \end{itemize}
      \end{lem}
      \begin{proof}
          As in the proof of \cite[Proposition~3.14]{Row26},
    $ag(a)(-)a_1 g(a) \in \mcA_0,$ so by hypothesis  $g(a)\in \mcA_0$.
      \end{proof}

\begin{lem}\label{indst}
  Suppose $(\mcA,\mcA_0,(-),\preceq)$ is a $(-)$-domain satisfying \eqref{pre2} and $a,a_i\in \tT.$
      If $f h\preceq \prod_{i=1}^t (\la^{k_i} (-) a_i) g$ is a partial $\preceq$-splitting of $f$, then any null root $a$ of  $f$ not satisfying $a^{k_i} = a_i$ for some $i$ is a null root of $g.$
\end{lem}
\begin{proof} Write $g_1 =  \prod_{i=2}^t (\la^{k_i} (-) a_i) g$. Then $fh \preceq \la^{k_1}g _1(-) a_1g_1$, so $\zero \preceq fh(a) \preceq a^{k_1}g_1(a) (-) a_1g_1(a),$ implying $g_1(a) \in \mcA_0.$
By induction on $t,$ $a$ is a null root of $g.$
  \end{proof}

 \begin{lem}\label{indsta} In a $(-)$-domain satisfying \eqref{pre2},
      if  $f \preceq (\la (-) a_1 )g$ and  $f \preceq (\la (-) a_1 )h$  for   tangible polynomials~$g,h$ then every  $a\ne a_1$ in $\tT$ is a null root of $g (-)h.$
 \end{lem}
\begin{proof}
    $f(a)  \preceq ag(a) (-)a_1 g(a),$ by \eqref{com3}, and similarly    $f(a)  \preceq ah(a) (-)a_1 h(a)$. Hence
 $$ \zero \preceq f(a)(-) f(a) \preceq  a(g(a)(-) h(a))(-) a_1(g(a)(-) h(a)),$$ implying $g(a)(-) h(a) \in \mcA_0.$\end{proof}

In order to treat uniqueness in partial $\preceq$-splittings, we need to address the converse of Lemma~\ref{indst0}: If $fh\preceq (\la^{k} (-)a_1) g$ and $a$ is a null root of $g,$  is~$a$ also a null root of~$f$?

\begin{rem} Suppose $f,g,h$ are tangible polynomials. \begin{enumerate} \eroman
    \item  If  $f\preceq(\la^k (-)a_1) g$ and $(\la^k (-)a_1) g$ is tangible then it equals $f,$ so obviously any null root of $g$ is a null root of $f.$

        \item If $fh \preceq g_1g_2$, then the product of $\preceq$-splittings of $g_1$ and $g_2$ yields a $\preceq$-splitting of~$f.$
\end{enumerate}
\end{rem}

\begin{MNote}\label{mulr} To simplify the exposition, we assume for the rest of this section that $(\mcA,\mcA_0,(-))$ is a $(-)$-domain with a $\tT$-reversible surpassing relation~$\preceq$. \begin{enumerate}\eroman
    \item By Theorem~\ref{Rev}(i) every  factor-root is a null root.

 \item In order to bypass   Counterexample~\ref{ce}, we also assume at times   that the Root Condition  holds. In view   of~Theorem~\ref{Rev}(ii), this includes tropicalizations of  semiring pairs, in particular supertropical pairs,  and the hypersystems of   hyperfields.
\end{enumerate}
  \end{MNote}

\begin{theorem}\label{spl2} Suppose  Hypothesis~R   holds, and $f$ has distinct   null roots $a_1, \dots, a_n$, where $n = \deg f.$
     Then   $f$ has a $\preceq$-splitting $f\preceq \prod_{i=1}^n (\la (-) a_i)$.  Furthermore, $a_1,\dots, a_n$ are the only null roots of $f.$
\end{theorem}

\begin{proof}
     Take a null root, and thus a factor-root, $a_1$ of $f$, and write $f \preceq (\la (-)a_1)g.$ Applying   induction on degree,   we have $g \preceq  \prod_{i=2}^n (\la (-) a_i)$. $f$ cannot have any other null root $a$, since then $a$ would be a null root of $g,$ contrary to the induction hypothesis. Furthermore,
     $n-1 = \deg g = \deg f -1,$ implying by induction that $a_2, \dots, a_n$ are all the null    roots of $g,$ so  $a_1, \dots, a_n$ are all the null roots of $f.$
\end{proof}

To circumvent  Hypothesis~R,
we could strengthen formally the Root Condition.

  \textbf{Factor-root Condition}. If  $f \preceq (\la (-) a_1 )g$ for $a_1\in \tT$ and $g$ a tangible polynomial, and
      $a\ne a_1$ is a   factor-root of $f$, then   $a$ is a   factor-root of $g$.

\begin{proposition}
    If   the Factor-root Condition   holds and  $f$ has $n$ \emph{distinct}  factor-roots $a_1, \dots, a_n,$ then $f$    $\preceq$-splits  uniquely with respect to $h=\one,$ as $f\preceq (\la (-) a_1)\dots (\la (-)a_n).$
\end{proposition}
\begin{proof}
    There are no other possible linear factors, by iteration. The order of the linear factors does not matter, by Lemma~
    \ref{nonb}.
\end{proof}


The following result puts metatangible pairs in the spotlight.

\begin{proposition}
\label{HypS}
If $(\mcA,\mcA_0)$ is metatangible, then  $f(a)\in \tT$ for almost all $a\in \tT.$
\end{proposition}
\begin{proof} We proceed  by induction on the number of monomials in $f.$
  In order for $f(a)\notin \tT,$ we would need some sub-polynomial $h$ of $f$ to have $h(a) \in \mcA_0,$ i.e., $a$ is a null root of $h.$ But there are at most $2^{n+1}$ proper sub-polynomials of $f$, each of which  has only finitely many null roots, by induction, so $f$ has only finitely many null roots.
\end{proof}

\subsubsection{Function partial splittings}$ $

 $\preceq$-splittings are too restrictive for our purposes.
We use the relation $\prect$ of Lemma~\ref{pol1}, and weaken the notion of $\preceq$-splitting.

\begin{lem}
    If $f_1\equiv f_2$ and $g_1\equiv g_2$, then $f_1 \prect(\la (-)a)g_1 $
    if and only if $f_2 \prect(\la (-)a)g_2 $.
\end{lem}
\begin{proof}
    $f_1(a_1) \preceq (a_1(-)a)g_1(a_1)$ for almost all $a_1\in \tT.$
\end{proof}

 \begin{definition}
  $ $\begin{enumerate}\eroman
      \item   A \textbf{function linear partial splitting}  of $f$ \textbf{with respect to $h$}  is  $$fh \prect \prod_{i=1}^t (\la (-)a_i)^{m_i}g,$$ for some $h\in \mcA[\la],$ where   $g$ is tangible with no null roots.

      \item
    $f$ is \textbf{$\prect$-split} if $f h \prect\prod_{i=1}^n(\la^{k_i}(-)a_i)$ for some $h\in \mcA[\la]$.
  \end{enumerate}
 \end{definition}

\begin{lem}\label{com2a} Suppose $(\mcA,\mcA_0,(-))$ is a $(-)$-domain.

\begin{enumerate}\eroman
     \item If $ (\la^k(-) a )g \in \mcA[\la]_0$, then  $g \in \mcA[\la]_0.$
\item If $\prect$ satisfies \eqref{pre2} and $f \prect (\la^k(-) a )g$ and  $f \prect (\la^k(-) a )h$  for   polynomials~$g,h$, then $g(-)h\in \mcA[\la]_0.$ In particular, if $(\mcA,\mcA_0)$ is metatangible, then $g\equiv h.$ \end{enumerate}
\end{lem}
\begin{proof} (i) For almost all $a_1\in \tT,$ $a_1^k\ne a,$ so $a_1^kg(a_1)(-)ag(a_1)\in \mcA_0,$ implying $g(a_1)\in \mcA_0.$

(ii)    $\zero\,  \prect  \,  f (-) f\,  \prect  \, (\la^k(-) a ) (g(-) h)$,  so $\zero\,  \prect  \, (\la^k(-)a)(g(-)h),$ implying $g(-)h \in \mcA[\la]_0.$  To see the last assertion, note using Proposition~\ref{HypS} that for any $a$ such that $g(a)$ and $h(a)$ are tangible, $g(a)=h(a).$
\end{proof}

\begin{theorem}\label{dr}
     Suppose       $(\mcA,\mcA_0,(-),\preceq)$   is a metatangible $(-)$-domain satisfying~\eqref{pre2}. Then, up to equivalence,   $f$ has a unique function linear partial splitting with respect to $h=\one$.
\end{theorem}

\begin{proof}
    We employ  a Jordan-H\"{o}lder-type argument. Take $f_1,f_2\equiv f,$ and $a_1,a_2\in \tT,$ satisfying
   $f_1\prect (\la(-)a_1)  g_1$ and $f_2\prect (\la(-)a_2)g_2$.
      If $a_1=a_2$ we apply Lemma~\ref{com2a}  and are done by induction on $\deg f.$ So assume that $a_1\ne a_2$. Then $ (\la(-)a_1)  \ |_{\prect} \ g_2$ and  $ (\la(-)a_2) \ |_{\prect} \ g_1$, so we write $$f_1 \prect (\la(-)a_1)  (\la(-)a_2)  h_1, \qquad  f_2\prect (\la(-)a_1) (\la(-)a_2)h_2.$$

Thus, by Proposition~\ref{HypS}, for  almost all $a\in \tT$, $h_1(a),h_2(a)\in \tT,$ and $$\zero \preceq f_1(a)(-)f_2(a) \preceq (a (-)a_1)(a (-)a_2)  h_1(a)(-)(a(-)a_1) (a (-)a_2)  h_2(a), $$
    implying $h_1(a) =  h_2(a)$; hence  $h_1 \equiv h_2 $, and we conclude by   induction on $\deg f$.
\end{proof}
\subsection{Repeated factor-roots}$ $

We turn to repeated roots.
In classical algebra, $a$ is called a double root of $f$ when $(\la (-) a)^2$ divides $f.$ We proposed several options in \cite{Row26}; here we proceed with the third. $f$ will always denote a  tangible polynomial  of degree $n$.

\begin{definition}\label{dbr}
     $a$ is a \textbf{double  factor-root} of $f$ if, for some tangible polynomial~$g$, $a$ is a  factor-root of $g$ and $f\preceq (\la (-) a)g.$

     Inductively, the  factor-root $a$ of $f$ has \textbf{multiplicity} $\ge m$ in $f$ if  $f\preceq (\la (-) a)g$ where $a$ has multiplicity $\ge m-1$ in $g.$
\end{definition}
\begin{example}\label{cea}
    As observed in Example~\ref{ce}, both $\one$ and $(-)\one$ are double  factor-roots of $f = \la^2+\la+\one$ in the polynomial hypersystem of $\mathcal{S}_w$, since $f\preceq (\la+\one)^2$ and $f\preceq (\la(-)\one)^2$.
\end{example}

\begin{lem} Suppose    the Factor-root Condition holds, and $a\in \tT$ is a factor-root of~$f$ of multiplicity $m.$
    \begin{enumerate}\eroman
        \item  $f \preceq (\la (-) a)^m g$  for some tangible polynomial~$g$ of degree $n-m$. \item If $f$ has $t$  factor-roots, counted with multiplicity, then $g$ has at least $t-m$  factor-roots.
    \end{enumerate}

\end{lem}
\begin{proof}
   The assertions follow from induction applied to the definition of multiplicity.
\end{proof}

These results motivate us to enlarge the system $(\mcA,\mcA_0,(-),\preceq)$ to include ``enough'' factor-roots of $f,$ cf.~\S\ref{ext} below.

\subsubsection{The Euclidean division algorithm}$ $

 Lemma~\ref{sub1} is a powerful tool in lifting classical theorems about polynomials to assertions about tangible polynomials in metatangible algebras. For instance, we are able to ``divide'' one polynomial by another, and obtain a remainder.

\begin{theorem}\label{EA}
    Suppose $(\mcA, \mcA_0)$ is metatangible and $\tT$ is a group.
Let  $f = \sum _{i=0}^n a_i \la^i$.
 For almost all     $a\in \tT$ there is a tangible polynomial $g = \sum_{i=0}^{n-1} b_i \la^i$  such that $f(a)\in \tT$ and $ f(\la)  \equiv (\la (-) a)g(\la) +f(a).$
\end{theorem}

\begin{proof} We use Proposition~\ref{HypS} repeatedly. First, we only consider those $a$ for which $f(a)\in \tT.$ We compute the coefficients $b_i\in \tTz$ of $g$ by means of \eqref{com3}.  $f(a)(-)a_0 \in \tT$ for almost all $a\in \tT;$ for these $a,$ put $b_0 =   (-)a^{-1}(f(a)(-)a_0)$. Inductively, we may assume that  $b_i:=(b_{i-1}(-)a_i)a^{-1}\in \tT.$ Now putting $h_a= (\la (-)a)g,$ we see by \eqref{com3} that $h_a$ is tangible for almost all $a\in \tT$. Hence,
\begin{equation}
    \begin{aligned}
        h_a(b)& = bg(b) (-) ag(b) = \sum_{i=0}^{n-1}(b b_i (-)a b_i)b^i \\ &= bb_0(-)(f(a)(-)a_0) +\sum_{i=1}^{n-1}(b b_i (-)(b_{i-1}(-)a_i))b^i \\& = bb_0 (-)(f(a)(-)a_0)+b_{n-1} b^n +\sum_{i=2}^{n} b_{i-1}b^i (-)\sum_{i=1}^{n-1} (b_{i-1}(-)a_i)b^i \\& = \sum_{i=1}^{n} b_{i-1}b^i (-)\sum_{i=1}^{n} b_{i-1}b^i +\sum _{i=0}^n a_i b^i (-)f(a) = (\sum_{i=1}^{n} b_{i-1}b^i )^\circ + f(b)(-)f(a).
    \end{aligned}
\end{equation}
For these $a,$ the polynomial $f(-)f(a)$ is tangible and has only finitely many null roots.
Thus $f(b)\ne f(a)$   for almost all $b\in \tT$, implying $h_a(b) = f(b)(-)f(a)$ whenever $h_a(b)$ is tangible, and thus $h_a(b) +f(a) = f(b)$ whenever $f(b)$ and $h_a(b)$ are distinct, by Lemma~\ref{sub1}, proving the theorem.
\end{proof}

\subsection{Ubiquity}\label{ubiquity1}
$ $

Throughout this subsection,  $f$ and $g$ denote tangible polynomials, and we assume  that the Root Condition holds.
Our next objective would be to show that if $f\equiv g,$   then
$f\cong g.$ This is too much to hope for, so we lower our sights.
 \begin{definition}\label{ubiq}
      A pair  $(\mcA,\mcA_0)$ satisfying the conditions of Note~\ref{mulr} is \textbf{$\tT$-ubiquitous} if whenever $f \equiv g$, we have   $f\equiv (f;g)\equiv g $.
 \end{definition}

  Ubiquity says   $f \equiv (f;g),$ so we may assume that $g = (f;g),$ i.e., $g$ is a sub-polynomial of $f.$

         \begin{ques}\label{ubeq}
      Suppose $f\equiv g$ and $f = g+\bar f,$ where $\bar f$ is tangible.  Under what conditions is $f\cong g$?
  \end{ques}

    This would obviously be the case if $f(-)g$ is tangible.

\begin{example}
    Here are some cheap counterexamples.
    \begin{enumerate}\eroman
\item The tropical extension of  \S\ref{ExP}, taking $\mcL = \Net,$ $\tT =\{\one\},$ and $\mcG$ arbitrary, written multiplicatively, i.e., the neutral element is written $\one$. Then for $f = \la^2 + \one$ and $g = \la,$ we have $f(\one,a)+g(\one,a)= f(\one,a)$ for all $a \ne \one \in \mcG, $ whereas $f(\one,\one) = (\one,\one)+(\one,\one)=(2,\one),$ and
        $$f(\one,\one) + g(\one,\one) = (\one,\one)+(\one,\one)+(\one,\one)=(3,\one).$$

    \item The supertropical pair $(\mathbb {T}(\tG),\tG )$ where $\tG$ is the max-plus algebra on $\Z.$ Then $f = \la^2  $ and $g = \la^2 +2\la  $ have the property that $f(a)=g(a)$ except for $a = 0,1 .$
      \end{enumerate}
\end{example}

  In the first counterexample,     $e+\one \ne e,$ whereas in the second counterexample, $\mcA$ is discrete.

\begin{proposition}
    \label{d2}  Suppose $(\mcA,\mcA_0)$   is $\tT$-ubiquitous.
    $ $\begin{enumerate}\eroman
        \item  Let $\{f_i : i\in I\}$ be a set of tangible polynomials, each equivalent to $f$. Then there is a
 polynomial $h\equiv f$  which is a common  sub-polynomial of each $f_i.$

  \item Up to equivalence, there is a unique common sub-polynomial of all the tangible polynomials equivalent to $f$.
    \end{enumerate}
\end{proposition}
\begin{proof}
    Start with the polynomial $h_1 = (f_1;f_2).$ If $ h_1$ is a sub-polynomial of all the $f_i$ we are done. If not, assume $h_1$ is not a sub-polynomial of $f_3,$ and take $h_2 = (h_1;f_3)$. Then $h_2$ is a sub-polynomial of $f_1,f_2,f_3$ and $h_2 \equiv f.$ Since a tangible  polynomial has only finitely many sub-polynomials, this process must terminate.

    (ii) By (i), we have a common tangible sub-polynomial $h$ equivalent to $f,$ and if $h_1$ is another such tangible polynomial then $h$ is a sub-polynomial of $h_1,$ by assumption.

\end{proof}

So for any equivalence class $\Phi$ of a tangible polynomial $f$,  we have a unique tangible  representative having no proper sub-polynomials in $\Phi$, which we denote as~$\tilde f$.

\begin{example}
  Taking  $f_\alpha = \la^2 +\alpha \la +\one$ over the supertropical pair $(\mathbb {T}(\tG),\tG )$, $\tilde{f_\alpha}\cong\la^2 +\one$ whenever $\alpha < \one$.
\end{example}

Here are two instances of~$\tT$-ubiquity.

\subsubsection{The metatangible case}$ $

\begin{lem}\label{com2}
  Suppose $(\mcA,\mcA_0,(-),\preceq)$   is a metatangible system.
\begin{enumerate}\eroman

 \item If  $ f(a) (-)g(a)\in \mcA_0$ for infinitely many $a\in\tT,$ then $f$ and $g$ have a common monomial.

  \item If  $f \prect (\la (-) a_1 )g$ and  $f \prect (\la (-) a_1 )h$  for   tangible polynomials~$g,h$, then $g$ and $h$ have a common monomial.

\end{enumerate}
 \end{lem}
 \begin{proof} (i)    Write $f = \sum a_i \lambda^i$ and $g = \sum b_i \lambda^i$ for $a_i,b_i\in \tT.$ By hypothesis, for infinitely many   $a\in \tT,$ $(f  (-)g)(a) \in \mcA_0,
     $ so the polynomial $f  (-) g$ cannot be tangible, i.e., $a_{i_0} (-)b_{i_0}\in \mcA_0$ for some $i_0,$ implying $a_{i_0} = b_{i_0}$.

     (ii) $\zero\,  \prect  \,  f (-) f\,  \prect  \, (\la (-) a_1 ) (g(-) h)$,  so $\zero\,  \prect  \, g(-)h,$ and apply (i).
\end{proof}

 \begin{lem}\label{com1}
     Suppose that $f,g$ are tangible polynomials.
  If $ f\equiv  g $, then writing $h= (f;g)$,  $\bar f (\pm)\bar g (\pm) h$   also are tangible.
 \end{lem}
 \begin{proof} Let $\bar f = \sum _{i\notin I} a_i \la^i$ and
     $ \bar g = \sum   _{i\notin I}  b_i \la^i$. Write $h:= (f;g) = \sum _{i\in I} a_i \la^i = \sum _{i\in I} b_i \la^i$. Note that $I\ne \emptyset$ by  Lemma~\ref{com2}(i). Then $\bar f (\pm)\bar g (\pm)h$ is tangible, since their supports are disjoint.
      \end{proof}

\begin{theorem}\label{dens} Suppose that $(\mcA,\mcA_0,(-))$ is a metatangible triple.
  \begin{enumerate}\eroman
      \item  $(\mcA,\mcA_0,(-))$ is   $\tT$-ubiquitous.
 \item (Addressing Question~\ref{ubeq}) Suppose  that $h_1\equiv h_2$ for tangible polynomials $h_1,h_2$ whenever $h_1(a)= h_2(a)$ for infinitely many $a\in \tT.$  If $f\equiv g$ for tangible polynomials $f,g$, then $f\cong g.$
  \end{enumerate}
\end{theorem}
\begin{proof}  Let $h = (f;g),$  with $f = \bar f +h$ and $g = \bar g +h,$ cf.~Definition~\ref{d1}.

(i) The assertion is tautological unless  $\tT$ is infinite.   Notation as in Definition~\ref{d1},
  let $\tT_1= \{ a\in \tT : f(a)=g(a)\in \tT,  \, h(a)\in \tT, \, \bar f (a) \ne \bar g(a) \},$ which is almost all of $\tT$, since $\bar f (-)\bar g$ is tangible by  Lemma~\ref{com1}. (Note that $f(a)(-)g(a)\in \mcA_0$ implies $f(a)= g(a)$ when $f(a),g(a)\in \tT.$)
We are done unless $f(a)\ne h(a)$ for infinitely many $a\in \tT_1$. Then for these $a,$ $f(a)(-)h(a) = \bar f(a),$ by Lemma~\ref{sub1}(ii), and likewise $g(a)(-)h(a) = \bar g(a),$ contrary to $ \tT_1$ being almost all of~$\tT.$

(ii)   For
almost all $a\in \tT$,    $f(a), \bar f(a), g(a), \bar g(a), h(a) \in \tT$ with $f(a)= g(a)$.   Let $\tT_1 = \{a \in \tT ; \bar f(a) = f(a)\}.$ If $\bar g(a) = f(a)$ for almost all $a\in \tT_1,$ then  by induction on the number of nonzero monomials in the polynomial, $\bar f \cong \bar g,$
 so $f =  \bar f +h \cong \bar g +h =g. $ Hence $ \bar g(a) \ne f(a)$ for infinitely many $a\in \tT_1,$ and for these~$a,$
  $\bar f(a) = f(a)(-) h(a) = g(a)(-)h(a) = \bar g(a)$. By hypothesis, $\bar f \equiv \bar g,$ so again, by induction $\bar f \cong \bar g,$ implying $f\cong g.$
\end{proof}

\begin{rem}
 The hypothesis of (ii) is a sort of density condition, which easily shows that Question~\ref{ubeq} has a positive answer for $(\mcA,\mcA_0,(-))$ supertropical (or more generally, of tropical type, as defined in \cite[Definition~4.6]{AGR2}) when $\tT$ is closed  under taking square roots.
\end{rem}

\section{Extensions of triples and systems}\label{ext}

 For the remainder of this paper, we assume that $\tT$ is infinite.

\begin{definition}$ $
   \begin{enumerate}
       \item  A  pre-triple   $(\mcA',\mcA'_0,(-))$ ~having underlying monoid $\tT'$ is an \textbf{extension} of a pre-triple
    $(\mcA,\mcA_0,(-)) $ ~having underlying monoid $\tT,$ if there is an injective homomorphism $\phi:\mcA \to   \mcA'$ such that $\phi(\mcA_0)\subseteq \mcA'_0,$   $\phi(\tT )\subseteq \tT' ,$ and $\phi((-)b) = (-)\phi(b)$ for all $b\in \mcA.$

    \item    A  pre-system   $(\mcA',\mcA'_0,(-),\preceq')$ ~having underlying monoid $\tT'$ is an \textbf{extension} of a system
    $(\mcA,\mcA_0,(-),\preceq) $
if $(\mcA',\mcA'_0,(-))$ is an extension of
    $(\mcA,\mcA_0,(-)) $, and $b_1\preceq b_2$ implies $\phi(b_1)\preceq' \phi (b_2)$.

     \item  A \textbf{systemic extension} is an extension of a system that is a system.

      \item  A \textbf{binomial extension} $(\mcA',\mcA'_0,(-))$ contains a null root of a binomial.
   \end{enumerate}
\end{definition}

 Given a monic polynomial $f$ over a system $(\mcA,\mcA_0,(-),\preceq)$, our goal in this section is to find a suitable  extension of $(\mcA,\mcA_0,(-),\preceq)$ in which   $f$ has a tangible root.

\begin{rem}
  \label{sp} Over a $(-)$-domain, a $\preceq$-splitting is an obstruction to adjoining tangible roots, for if $f \preceq \prod (\la^{k_i} (-)  a_i) $ and $f(a)\in \mcA_0$ for $a\in \tT,$ then $a^{k_i} = a_i$ for some $i.$
\end{rem}

\begin{example} The following extensions are relevant:
    \begin{enumerate}\eroman
 \item The \textbf{functional polynomial extension}
$(\mcA[\la],\mcA [\la]_0,(-),\prect)$ of a system
    $(\mcA,\mcA_0,(-),\preceq) $.

\item  Any $\tT$-submodule $\mcA'_0 \supseteq \mcA_0$ yields the \textbf{null extension}  $(\mcA,\mcA'_0,(-))$ of a pre-triple
    $(\mcA,\mcA_0,(-)) $, where $\phi$ is the identity injection.

    \end{enumerate}
\end{example}

Note that a null extension of a triple   may fail to be uniquely negated, and $\preceq$ remains a pre-surpassing relation, but could fail the condition of Definition~\ref{surp}(iii).

We want to modify the  functional polynomial extension to make it uniquely negated. Recall the equivalence $\equiv$ of Definition~\ref{null1}.

\begin{definition}  $(\mcA[\la] /\equiv , \mcA[\la]_0/\equiv) $ is called the  \textbf{reduced polynomial extension} of a pair $(\mcA  , \mcA _0) $.   $[f]$ denotes the image of $f$ in $\mcA[\la] /\equiv$.
\end{definition}

%

\begin{lemma}\label{ess1}  The reduced polynomial extension  $(\mcA[\la] /\equiv , \mcA[\la]_0/\equiv)$  is a pair  with   underlying monoid $\tT_\la$, under the action $a\la^i[f] = [a\la^if]$.
 If $(\mcA  , \mcA _0) $ has a negation map $(-)$, then  $(\mcA[\la] /\equiv , \mcA[\la]_0/\equiv)$ has the induced negation map.
\begin{enumerate}\eroman
\item The pre-surpassing relation    $\prect$ on $(\mcA[\la]   , \mcA[\la]_0)$  induces a pre-surpassing relation $\preceq_\equiv$ on $(\mcA[\la] /\equiv , \mcA[\la]_0/\equiv)$.

\item  If   $(\mcA , \mcA _0,(-)) $  is uniquely negated, then  $(\mcA[\la] /\equiv , \mcA[\la]_0/\equiv)$ is uniquely negated.

 \item  If  $(\mcA , \mcA _0,(-),\preceq) $  is a
system, then $(\mcA[\la] /\equiv , \mcA[\la]_0/\equiv,(-),\preceq_\equiv)$ is a
system.
    \end{enumerate}
\end{lemma}
\begin{proof}  $(\mcA[\la] /\equiv , \mcA[\la]_0/\equiv)$ is a pair, seen by taking tangible substitutions. Likewise, if $f\equiv g,$ then $(-)f \equiv (-)g,$ so $(-)$ is well-defined, implying $(-)[f] = [(-)f]$.

  (i)  $\preceq_\equiv$ is seen to be a preorder by checking tangible evaluations. If $g\in \mcA[\la]_0$ then $\zero \preceq g(a)$ for almost all $a\in \tT,$ implying $\zero \prect g$ and thus $[\zero] \preceq_\equiv [g]$.

(ii)  For almost all $a\in \tT,$ if $f(a)+g(a)\in \mcA_0$ for $f(a),g(a)\in \tT,$ then $g(a) =(-)f(a),$ so $[g]=(-)[f].$

(iii) $(\mcA[\la] /\equiv , \mcA[\la]_0/\equiv,(-))$
is a triple by (ii); $\preceq_\equiv$ is a surpassing relation, by Lemma~\ref{sur1}(ii).
\end{proof}

\begin{theorem}\label{pol1a2} Suppose that the pre-triple $(\mcA,\mcA_0,(-))$ is a $(-)$-domain  with a pre-surpassing relation~$\preceq $, and    $f$ is
 a polynomial which is not $\prect$-split. Suppose furthermore that $\tT$ is a group.
Then  $(\mcA,\mcA_0,(-),\preceq)$ has an extension which we denote by $(\mcA,\mcA_0,(-),\preceq)_f $   with underlying monoid $\tT_\mu := \{ a\mu^i: a\in \tT,\ i \in \Net\},$   over which $f$ has a null root.
\begin{enumerate}\eroman
\item If $(\mcA,\mcA_0,(-),\preceq)$ satisfies \eqref{pre2}, then so does $(\mcA,\mcA_0,(-),\preceq)_f.$
\item
If $(\mcA,\mcA_0,(-),\preceq)$ is uniquely negated, then so is
$(\mcA,\mcA_0,(-),\preceq)_f$, which thus is a system.

\item
$(\mcA,\mcA_0,(-),\preceq)_f$ is a $(-)$-domain.

\item If the surpassing relation $\preceq$ is $\tT$-reversible on $(\mcA,\mcA_0,(-),\preceq)$ then $\preceq$ is $\tT_\mu$-reversible on $(\mcA,\mcA_0,(-),\preceq)_f. $

\item If $\preceq$ is strongly $\tT$-reversible on $(\mcA,\mcA_0,(-),\preceq)$ then $\preceq$ is strongly $\tT_\mu$-reversible on $(\mcA,\mcA_0,(-),\preceq)_f. $

 \item If $\mcA$ is a semiring, then $\mcA_f$ is a semiring.

\end{enumerate}

\end{theorem}
\begin{proof}  
 We may assume that $f$ has no null roots, and we embed $(\mcA,\mcA_0,(-),\preceq)$ into the reduced  polynomial extension  system $(\mcA[\mu] /\equiv , \mcA[\mu]_0/\equiv,(-),\preceq_\equiv)$,  for a commuting indeterminate $\mu$,  with underlying monoid $\tT_\mu,$ and the surpassing relation of  Lemma~\ref{ess1}.
 Put $\mcA_f = \mcA[\mu]/\equiv$. The subsequent definitions are seen to be independent of the representative $h$ of $[h]$ since any two representatives agree on almost all $a\in \tT.$ We extend the null set to be
  \begin{equation}\label{nulf}
      \begin{aligned}
{\mcA_f}_0 := \bigg\{ [h] &\in \mcA_f:   \text{ For some  } a_{i,j}\in \tTz,\ k_i\in \Net, \, \ g\in \mcA[\mu],\\ & g  f \prect  \prod_{i=1}^t (a_{i,1}\mu^{k_i}(-)a_{i,2})h \bigg\},
      \end{aligned}
  \end{equation}  ${\mcA_f}_0 $ is a submodule over the same underlying monoid $\tT_\mu,$
  since $$ g  f \prect\prod_{i=1}^t (a_{i,1}\mu^{k_i}(-)a_{i,2})h, \qquad g'f  \prect \prod_{i=1}^t (a'_{i,1}\mu^{k'_i}(-)a'_{i,2})h' $$ imply \begin{equation}\label{sum1}
      \begin{aligned}\bigg(\prod_{i=1}^t (a_{i,1}\mu^{k_i}&(-)a_{i,2})g +\prod_{i=1}^t (a'_{i,1}\mu^{k'_i}(-)a'_{i,2})g'\bigg)f \\ &\prect \prod_{i=1}^t (a_{i,1}\mu^{k_i}(-)a_{i,2})\prod_{i=1}^t (a'_{i,1}\mu^{k'_i}(-)a'_{i,2})(h+h').
      \end{aligned}
  \end{equation}
  We put
$(\mcA,\mcA_0,(-),\preceq)_f :=(\mcA_f,{\mcA_f}_0,(-),\preceq).$

The pre-surpassing relation
on $(\mcA,\mcA_0,(-),\preceq)$ extends to a pre-surpassing relation $\preceq$
on $(\mcA ,\mcA _0)_f$, given by $[h_1] \preceq [h_2]$ if there are $a_{i,j}\in \tT$, $g(\mu)\in \mcA[\mu]$, such that \begin{equation}
    \label{pr1} \prod_{i=1}^t (a_{i,1}\mu^{k_i}(-)a_{i,2}) h_1 +gf\prect \prod_{i=1}^t (a_{i,1}\mu^{k_i}(-)a_{i,2})h_2.
\end{equation} Indeed for $[h]\in  {\mcA_f}_0,$   writing $g  f \prect \prod_{i=1}^t (a_{i,1}\mu^{k_i}(-)a_{i,2})h$  for $k_i\in \Net,$
we see for almost all $a\in \tT$ that  $\zero \preceq h(a) ,$ so  $\zero \preceq  [h] .$
Likewise if $h_i\preceq h_i'$ for $i = 1,2, $ then $h_1+h_2 \preceq h_1'+h_2',$ seen by applying \eqref{pr1} twice and adding. Transitivity is proved  analogously to \eqref{sum1}.

    The pre-surpassing relation $\preceq$ on   $(\mcA   , \mcA _0)_f$   is a surpassing relation,  because if $[a_1 \mu^\ell] \preceq [a_2 \mu^m]$ then for some polynomial $g(\mu),$ $$\prod_{i=1}^t (a_{i,1}a^{k_i}(-)a_{i,2})a_1a^\ell +g(a)f(a)  \preceq a_2 a^m\prod_{i=1}^t (a_{i,1}a^{k_i}(-)a_{i,2})$$ for almost all $a$ in~$ \tT,$  implying $$g(a)f(a) \preceq a_2a^m \prod_{i=1}^t (a_{i,1}a^{k_i}(-)a_{i,2})(-)a_1a^\ell \prod_{i=1}^t (a_{i,1}a^{k_i}(-)a_{i,2}),$$ so $gf \prect (a_2 \mu^m (-) a_1 \mu^\ell)\prod_{i=1}^t (a_{i,1}\mu^{k_i}(-)a_{i,2}),$ implying $f$ is $\prect$-split, a contradiction.

(i) If $\zero \preceq [h]$ then $[h]\in {\mcA_f}_0,$ by \eqref{pr1}.

(ii) If $[a_1'\mu^i] + [a'_2\mu^j] \in {\mcA_f}_0 $ with $a_1'\mu^i \ne a'_2\mu^j $, $i\ge j,$ then for $a_{i,j}\in \tT$, $g (\mu)\in \mcA[\mu]$,  $gf \prect \prod_{i=1}^t (a_{i,1}\mu^{k_i}(-)a_{i,2}) (a_1'\mu^i + a'_2\mu^j),  $  yielding $$(a_1')^{-1}gf \prect  ( \mu^{i-j}+a_1'^{-1}a'_{2})\mu^{j} \prod_{i=1}^t (a_{i,1}\mu^{k_i}(-)a_{i,2})  .$$ If $i = j$ then we can cancel $a_1'+a_2'$ by specializing $\mu$ to almost all $a\in \tT,$ since $(\mcA,\mcA_0)$ is a $(-)$-domain. If $i>j$, $f$ already is $\prect$-split, by inspection.  Hence $f$ is $\prect$-split, a contradiction.

For the proofs of $\tT_\mu$-reversibility and strong $\tT_\mu$-reversibility, just mimic the proofs of Lemma~\ref{pol1}(ii).

(iii) Suppose $[a'_1  \mu^i h(\mu) ](-)[a'_2\mu^j h(\mu)] \in {\mcA_f}_0 .$  Then $$ g  f \prect (a'_1  \mu^i (-)a'_2\mu^j )\prod_{i=1}^t (a_{i,1}\mu^{k_i}(-)a_{i,2})h  $$ for  $a_{i,j}\in \tT,$  $g \in \mcA[\mu]$.       If, say, $i \ge j$ then $a'_1  \mu^i  (-)a'_2\mu^j =a'_1  (\mu^{i-j}  (-){a'_1}^{-1}a'_2 )\mu^{j}. $ Thus $[h]\in {\mcA_f}_0,$ since we merely have increased $t$ in \eqref{nulf}.

    (iv) $\tT_\mu$-reversibility:   If  $(-)[a_0 \mu^{i_0}] \preceq [\sum_{j=1}^n a_j\mu^{i_j}]$  then, for  almost all $a\in \tT$, \begin{equation}
        \prod_{i=1}^t (a_{i,1}a^{k_i}(-)a_{i,2}) (-)a_0 a^{i_0} +g(a)f(a) \preceq \prod_{i=1}^t (a_{i,1}a^{k_i}(-)a_{i,2})\sum_{j=1}^n a_j a^{i_j} +g(a)f(a),
    \end{equation} which implies $$\prod_{i=1}^t (a_{i,1}a^{k_i}(-)a_{i,2})(-)a_1 a^{i_1} \preceq  \prod_{i=1}^t (a_{i,1}a^{k_i}(-)a_{i,2})(a_0 a^{i_0} + \sum_{j=2}^n a_j a^{i_j})+g(a)f(a),$$ and thus $(-)a_1 \mu^{i_1} \preceq  a_0  \mu^{i_0}+ \sum_{j=2}^n a_j \mu^{i_j}.$

(v) Strong $\tT_\mu$-reversibility: If  $[a_1 \mu^i ]+[h(\mu)]\in {\mcA_f}_0 $ for $a_1 \in\tT$ then, for almost all $a\in \tT,$ \begin{equation}
    \prod_{i=1}^t (a_{i,1}a^{k_i}(-)a_{i,2}) ((-)a_1   a^i + h(a)) +g(a)f(a)\in \mcA_0,
\end{equation}  so $\prod_{i=1}^t (a_{i,1}a^{k_i}(-)a_{i,2})((-)  a_1 a^i )\preceq \prod_{i=1}^t (a_{i,1}a^{k_i}(-)a_{i,2}) h(a)+g(a)f(a),$ implying $(-)[a_1\mu^i] \preceq [h(\mu)]$.

(vi)  $\mcA_f$ is a semiring, via convolution multiplication.
\end{proof}

Although $(\mcA,\mcA_0,(-),\preceq)_f$ is $\preceq$-f.s.,
lacking an answer to Question~\ref{fs1}, we look for an alternative solution for binomials.

\begin{theorem}\label{pol1a1} Suppose  $(\mcA,\mcA_0,(-),\preceq)$ is a $(-)$-domain with a pre-surpassing relation $\preceq $,
and the polynomial   $f = \la^n (-) a\la^j$ does not have any null roots. Then  $(\mcA,\mcA_0,(-),\preceq)$ has an extension $(\mcA,\mcA_0,(-),\preceq)_f $ which is also  a   $(-)$-domain,  over which $f$ has a null root which is also  a factor-root,  and $\preceq $ extends to a pre-surpassing relation on $(\mcA,\mcA_0,(-),\preceq)_f $. Furthermore,
\begin{enumerate}\eroman

\item If $(\mcA,\mcA_0,(-),\preceq)$ satisfies \eqref{pre2}, then so does $(\mcA,\mcA_0,(-),\preceq)_f.$

\item
If $(\mcA,\mcA_0,(-),\preceq)$ is uniquely negated, then
$(\mcA,\mcA_0,(-),\preceq)_f$ is  uniquely negated, and thus a system.

\item If the surpassing relation $\preceq$ is $\tT$-reversible on $(\mcA,\mcA_0,(-),\preceq)$ then $\preceq$ is $\tT_\mu$-reversible on $(\mcA,\mcA_0,(-),\preceq)_f. $

\item If $\preceq$ is strongly $\tT$-reversible on $(\mcA,\mcA_0,(-),\preceq)$ then $\preceq$ is strongly $\tT_\mu$-reversible on $(\mcA,\mcA_0,(-),\preceq)_f. $
 \item If $\mcA$ is a semiring, then $\mcA_f$ is a semiring.
\end{enumerate}
\end{theorem}

\begin{proof} As in the proof of Theorem~\ref{pol1a2},
 we shall formally adjoin a null root to~$f.$ Towards this end, we first embed $(\mcA,\mcA_0,(-),\preceq)$ into the  polynomial extension $(\mcA[\mu],\mcA[\mu]_0,(-),\preceq),$ for a commuting indeterminate $\mu$,  with underlying monoid $\tT_\mu = \{ a_i\mu^i: a_i\in \tT,\ 0\le i <n\},$ defining $\mu^n = a,$ and the surpassing relation of  Lemma~\ref{pol1}.
 If $j>0$ then $\zero$ is both a $\preceq$-root and a factor-root, so we may assume $j =0,$ i.e., $f = \la^n (-) a.$ Define $\mcA_f$ by identifying $\mu^n$ with~$a,$ i.e.,   $\mu^i \mu^j = a \mu^{i+j-n}$ when $i+j\ge n $, and $\mcA_f = \oplus _{i=0}^{n-1} \mcA \mu^i,$ with the $\tT_\mu$-action  given by $$(a_1\mu^i)\left(\sum b_j \mu^j\right) = \sum _{i+j <n} a_1 b_j \mu^{i+j} + \sum _{i+j \ge n} a_1 b_ja \mu^{i+j-n}  ,$$ for $a_1\in \tT,$ $b_j \in \mcA,$ and  ${\mcA_f}_0 = \oplus _{i=0}^{n-1} \mcA_0 \mu^i.$     Clearly $({\mcA_f},{\mcA_f}_0)$ is a pre-triple with underlying monoid $\tT_\mu  $  and having   negation map $(-)\sum a_i \mu^i = \sum ((-)a_i)\mu^i$, and surpassing relation $\sum b_i \mu^i \preceq \sum b_i' \mu^i$ when each $b_i\preceq b_i'$.

(i) is seen componentwise.

  (ii),  (iii), and (iv):  The uniquely negated property, $\tT_\mu$-reversibility and strong  $\tT_\mu$-reversibility pass componentwise to the extension.
      Obviously $\mu$ is both a $\preceq$-root and a factor-root of $f.$

   (v)   When $\mcA$ is a semiring, multiplication in $\mcA[\mu]$ is given by
 $$\left(\sum b_i \mu^i\right)\left(\sum b'_j \mu^j\right) = \sum _{i+j <n} b_ib_j' \mu^{i+j} + \sum _{i+j \ge n} b_ib_j'a \mu^{i+j-n}  .$$
\end{proof}

\begin{rem}$ $ \begin{enumerate}\eroman

   \item Remark~\ref{sp} shows that when $\tT$ is closed under taking radicals, we cannot hope to adjoin null roots to $\preceq$-split polynomials. This makes Question~\ref{ubeq} even more relevant.

   \item In order to repeat this procedure, we need to know that if   $(\mcA,\mcA_0,(-),\preceq)$ is a $(-)$-domain, then so is
$(\mcA,\mcA_0,(-),\preceq)_f$. See Example~\ref{fs2} for the quadratic case.
\end{enumerate}
\end{rem}

 \subsection{$\preceq$-Algebraic and $\preceq$-integral extensions}$ $

\begin{definition}  Suppose $(\mcA',\mcA'_0,(-),\preceq)$ is an extension of $(\mcA,\mcA_0,(-),\preceq)$, with underlying monoid $\tT'$.\begin{itemize}
\item   An element   $a'\in \tT'$ is \textbf{algebraic} of degree $n$ over $(\mcA,\mcA_0,(-),\preceq)$ if it is a null root  of some tangible  $f\in \mcA[\la]$ of  degree $n$;  $a'$ is \textbf{integral} over $(\mcA,\mcA_0,(-),\preceq)$ if $f$ can be taken monic.

   \item      An element $b\in \mcA'$  is $\preceq$-\textbf{integral} (with $\deg b = n$) over $(\mcA,\mcA_0,(-),\preceq)$ if $b$ is a $\preceq$-root of a tangible monic polynomial $f\in \mcA[\la]$  (of minimal possible degree $n$).

    \item $(\mcA',\mcA'_0,(-),\preceq)$ is $\preceq$-\textbf{integral} over $(\mcA,\mcA_0,(-),\preceq)$ if each   element of $\tT'$ is
$\preceq$-integral   over $\mcA.$
    \end{itemize}
\end{definition}

\begin{rem} When $\tT$ is a group, ``algebraic'' and ``integral'' are the same, but, as in classical algebra, the general theory flows more smoothly for ``integral,'' so we are switching our focus to integrality.
\end{rem}

\begin{lem}\label{revi}
  If $(\mcA',\mcA_0',(-),\preceq)$ is either   strongly $\tT'$-reversible or  metatangible, and if $a' \in \tT'$ is integral over $(\mcA,\mcA_0,(-),\preceq)$, then  $a'$ is $\preceq$-integral over $(\mcA,\mcA_0,(-),\preceq)$.
\end{lem}
\begin{proof}
 Suppose $\sum_{i=0}^n a_i {a'}^i \in \mcA_0'$. When $(\mcA',\mcA_0',(-),\preceq)$ is strongly $\tT'$-reversible, the assertion   holds by definition. For metatangible, we may assume that no sub-sum is in $\mcA_0',$ since otherwise we conclude by induction on the number of monomials. But then $\sum_{i=0}^{n-1} a_i {a'}^i \in \tT'$, so $a_n {a'}^n = (-) \sum_{i=0}^{n-1} a_i {a'}^i$, since   $(\mcA',\mcA'_0)$ is uniquely negated.
\end{proof}

\begin{definition}       A system $(\mcA,\mcA_0,(-),\preceq)$ is \textbf{$\preceq$-integrally closed} if every monic tangible non-$\prect$-split polynomial over $(\mcA,\mcA_0,(-),\preceq)$ has a  $\preceq$-root.
\end{definition}

Examples of $\preceq$-integrally closed $(-)$-domains include the supertropical triple over a divisible group (such as $(\Q,+)$), the Krasner hyperfield, the  hyperfield of signs, and the  hyperfield of weak signs.
The phase hyperfield does not have a null root of $\la^2 +\one.$

 We write $S_1 \preceq S_2$ for subsets of $\mcA$ to indicate that for each $s_1\in S_1$ there is $s_2\in S_2$ for which $s_1\preceq s_2.$
\begin{definition}\label{fs} Suppose $(\mcA',\mcA'_0,(-),\preceq)$ is  an extension   of $(\mcA,\mcA_0,(-),\preceq)$ having underlying monoid $\tT'$.
\begin{enumerate}
    \item
   A subset $S$ of a {$\tT'$-module} $\mcM$  is \textbf{finitely spanned} (f.s.) by $n$ elements $a'_1, \dots, a'_n \in \tT'$ if   $S = \sum_{i=1}^n \tT a'_i$.

  \item
   A subset $S$ of a {$\tT'$-module} $\mcM$  is $\preceq$-\textbf{finitely spanned} ($\preceq$-f.s.) by $n$ elements $a'_1, \dots, a'_n \in \tT'$ if   $S \preceq \sum_{i=1}^n \tT a'_i$.

   \item  An extension $(\mcA',\mcA_0',(-),\preceq)$  is $\preceq$-\textbf{locally f.s.}~over $(\mcA,\mcA_0,(-),\preceq)$ if for each   $a'\in \tT',$ $\tT'[a']$ is $\preceq$-{finitely spanned}.

    \item  An extension $(\mcA',\mcA_0',(-),\preceq)$  is $\preceq$-\textbf{f.s.}~over $(\mcA,\mcA_0,(-),\preceq)$ if $\tT'$ is $\preceq$-{finitely spanned}.

\end{enumerate}
\end{definition}

\begin{rem}\label{trans} Let $(\mcA',\mcA'_0,(-),\preceq)$ be  an extension   of $(\mcA,\mcA_0,(-),\preceq)$ with  underlying monoid $\tT'$, and
$(\mcA'',\mcA''_0,(-),\preceq)$ be a systemic extension of $(\mcA',\mcA'_0,(-),\preceq)$  with  underlying monoid $\tT''$.
 \begin{enumerate}\eroman
     \item If $S \subseteq \mcA''$ is f.s.~over $\tT'$ and $\tT'$ is f.s.~over $\tT,$ then $S$  is f.s.~over $\tT.$

        \item If $S \subseteq \mcA'$ is $\preceq$-f.s.~over $\tT'$ and $\tT'$ is $\preceq$-f.s.~over $\tT,$ then $S$  is $\preceq$-f.s. over~$\tT.$

      \item
 Any $\preceq$-integral element $a'$ of~$\tT'$ of degree $n$ satisfies ${a'}^n \preceq \sum _{i=0}^{n-1}\tT {a'}^i$, which is  $\preceq$-f.s.
 \end{enumerate}
\end{rem}

\begin{proposition}\label{dep1}
   Suppose $(\mcA',\mcA'_0,(-),\preceq)$ is  an extension~of $(\mcA,\mcA_0,(-),\preceq)$.
   \begin{enumerate}\eroman
       \item If $(\mcA',\mcA'_0,(-),\preceq)$ is  $\preceq$-integral, then for each   $a'\in \tT',$ $\tT[a']$ is  $\preceq$-f.s.

  \item In the other direction, if $(\mcA,\mcA_0)$ is metatangible   or $(\mcA,\mcA_0)$ is  a hyperpair, and a $\tT$-submodule $W$ of $\mcA'$ is spanned by $n$ elements $\bfv_1,\dots, \bf_n$ and $\tT[a']\subseteq W$, then $ a'$ is algebraic of degree $ \le 2^ n.$
   \end{enumerate}
  \end{proposition}
\begin{proof}
    (i) Take $n = \deg a'$ and $a'_i = {a'}^{i-1},$ $1\le i \le n.$

       (ii)
 In the classical case, one can use the Cayley-Hamilton polynomial of the multiplication map of $a'.$ However,  it is not guaranteed to be tangible, so we need a trickier argument.

Let  $V = \mcA^{(n)}$ and $V_0 = \mcA_0^{(n)}$.
 The \textbf{vector space pair} $(V,V_0)$ has componentwise addition and scalar multiplication $a(a_1,\dots,a_n) = (aa_1,\dots, aa_n)$. A set of vectors $\{\bfw_1,\dots, \bfw_m\}$ is \textbf{dependent} if $\sum_{i=1}^m a_i \bfw_i \in V_0$ for suitable $a_i\in \tTz$ not all~$\zero.$

 \cite[Corollary~7.53]{AGR2} says that if  $(\mcA,\mcA_0)$ is metatangible   or $(\mcA,\mcA_0)$ is  a hyperpair, then any $2^n$ vectors of $V$ are dependent. In particular, writing $(a')^j = \sum _{i=1}^n a_{i,j}\bfv_i,$ the vectors $(a_{1,j},\dots, a_{n,j}),  $ $1\le j \le 2^n,$ are dependent,  i.e., there are $a_j$ not all $\zero$ such that $\sum a_j a_{i,j}\in \mcA_0.$ But then $\sum_{j=1}^{2^n} a_j (a')^j = (\sum _i (\sum_j a_j   a_{i,j} )\bfv_i) \in \mcA_0 V \subseteq \mcA'_0,$ proving $a'$ is algebraic of degree $\le 2^n.$
\end{proof}

\begin{proposition}\label{int2} Suppose that  $(\mcA',\mcA'_0,(-),\preceq)$ is a systemic extension of $(\mcA,\mcA_0,(-),\preceq)$, and
$(\mcA'',\mcA''_0,(-),\preceq)$ is a systemic extension of $(\mcA',\mcA'_0,(-),\preceq).$  If $(\mcA',\mcA'_0,(-),\preceq)$ is   $\preceq$-integral over $(\mcA,\mcA_0,(-),\preceq)$ and  $a''\in \tT''$ is $\preceq$-integral  over $(\mcA',\mcA'_0,(-),\preceq),$ then  $\tT[a''] $ is $\preceq$-f.s. over  $(\mcA,\mcA_0,(-),\preceq).$
\end{proposition}
    \begin{proof}  The proof proceeds along classical lines.  Namely, if $a'' \in \tT''$ then write
    ${a''}^n \preceq \sum_{i=0}^{n-1} a'_i {a''}^i $ for $a_i'\in \tT'.$ Then
    $\tT[a_i'] \preceq  \sum_{j=1}^{n_i} \tT a'_{i,j}$ for suitable $a'_{i,j}\in \tT',$ implying $\tT[a''] \preceq  \sum_{i=0}^{n-1} \sum_{j=1}^{n_i} \tT  a'_{i,j}{a''}^i$.
    \end{proof}

We would like to conclude transitivity of
    $\preceq$-integral extensions, using Proposition~\ref{dep1}(ii), but have the following obstacle:
 \begin{ques}\label{fs1}
    Under what conditions on a system are $\preceq$-f.s. vector spaces f.s.?
\end{ques}

\subsection{The strongly $\tT$-reversible case over an underlying group $\tT$}
\begin{theorem}\label{FT1}  Suppose that  $(\mcA,\mcA_0,(-))$ is a $(-)$-domain obtained by a chain of $\preceq$-f.s.~extensions,  and $\preceq$ is strongly $\tT$-reversible on $(\mcA,\mcA_0,(-),\preceq)$,  for an infinite group $\tT.$ Then, over some underlying monoid $\bar \tT,$ the system $(\mcA,\mcA_0,(-),\preceq)$  has a strongly $\bar \tT$-reversible, locally $\preceq$-f.s.~systemic extension $\overline{ (\mcA,\mcA_0,(-),\preceq) }$ which is a $(-)$-domain, in which every non-$\prect$-split  tangible polynomial has a $\preceq$-root.
\end{theorem}
\begin{proof}   We use a transfinite induction argument which has become  standard in the  usual algebraic framework, building on Theorem~\ref{pol1a2}. Namely, for each monic non-$\prect$-split tangible polynomial $f$ over $(\mcA,\mcA_0,(-),\preceq)$ we can find a strongly $\tT$-reversible extension $(\mcA,\mcA_0,(-),\preceq)_f $  in which $f$ has a null root, which is a $\preceq$-root by Lemma~\ref{revi}. Then,  by Proposition~\ref{dep1}(i), $(\mcA,\mcA_0,(-),\preceq)_{f} $  is $\preceq$-f.s. over~$\tT$.  We continue, localizing via Lemma~\ref{loc} to get an underlying group; then we form  a chain of $\preceq$-f.s. extensions,   and take their union, which is a locally $\preceq$-f.s.~extension.  Strong reversibility of the  surpassing relation and the fact that the extension of $(\mcA, \mcA_0, (-), \preceq)$
is a $(-)$-domain are immediate since we are taking a direct limit.

       In the process we may have introduced additional non-$\prect$-split tangible polynomials, since the underlying monoid has also been extended;   we conclude by using transfinite induction.
\end{proof}

We have not claimed that every element of $\overline{ (\mcA,\mcA_0,(-),\preceq) }$ is $\preceq$-integral.

 \begin{ques}
     Under what conditions is $\overline{ (\mcA,\mcA_0,(-),\preceq) }$ a $\preceq$-integral extension of $(\mcA,\mcA_0,(-),\preceq)$?
     \end{ques}

In other words, are the extraneous elements in the extension also $\preceq$-integral?

\section{Real pairs}
\label{real0}

\begin{definition}\label{real1}  $ $ \begin{enumerate}
    \item  A monoid $\tT$ is \textbf{real} if $\tT$ is a disjoint union $\tT = \tT_+ \cup \tT_-$ for some   submonoid $\tT_+$ of $\tT$, called the \textbf{positive submonoid} of $\tT$, with  $\tT_- $ containing an element    $(-)\one \in \tT_-$ satisfying $\tT_- = ((-)\one)\tT_+$ and $((-)\one)^2 = \one.$

  \item Given  $\tT$ real, we define the \textbf{absolute value} $\abs{ \phantom{w}}:\tT\to \tT_+$ by
$$\abs{a} = \begin{cases}
    a: \ a\in \tT_+,\\    (-)a: \ a\in \tT_-.
\end{cases}$$ We do not define $\abs{ \phantom{w}}$ more generally, out of concern that it would not be well defined.

    \item    A pair     $(\mcA,\mcA_0)$ is \textbf{tangibly real}   if its underlying monoid $\tT$ is real.
\end{enumerate}
\end{definition}

\begin{example}
    Suppose $\mcA$ is a $\tT$-module.
 The
      doubled pair $(\widehat{\mcA},\widehat{\mcA}_0)$  is tangibly real, with $\hat{\tT}_+ =  \tT\times \{\zero\}$ and $\hat{\tT}_- = \{\zero\} \times\tT $. Also, $\mcA_+ = \mcA \times \{\zero\},$ and $\mcA_- = \{\zero\} \times \mcA  $.
\end{example}



  \subsection{Metatangible pairs over a real submonoid}$ $

From now on, suppose that the underlying monoid $\tT$ is  real.

\begin{lem} Suppose that $(\mcA,\mcA_0)$ is a pair over $\tT.$
\begin{enumerate} \eroman
    \item  $(\mcA,\mcA_0)$ has a   negation map $(-)$ of the second kind, given by $(-)b = ((-)\one)b$.

   \item  Suppose $(\mcA,\mcA_0)$ also is metatangible and uniquely negated.
\begin{enumerate} \eroman
    \item  $(\mcA,\mcA_0)$ is tangibly real. Furthermore, putting $\mcA_+ = \tT_+^\circ \cup \tT_+$ and  $\mcA_- =  \tT_+^\circ \cup \tT_-$, we have $\mcA = \mcA_+ \cup \mcA_-,$ with  $\mcA_+\cap \mcA_- = \mcA_0 = \tT_+^\circ$.

       \item   $a+e\in \{a, e\}$ for all $a\in \tT.$
    In particular,
    $\one + e$ is either $\one$ or~$e$.
\end{enumerate}
\end{enumerate}
\end{lem}
 \begin{proof} (i) $(-)$ is clearly a negation map. One cannot have   $\one =(-)\one  \in \tT_+\cap \tT_-.$

(ii)
     (a)  By (i),   $(\mcA,\mcA_0)$ is of the   second kind. Hence $\mcA_0 \subseteq \tT_+ +\tT_-$, and $\mcA = \mcA_+ \cup \mcA_- $  follows from Lemma~\ref{1p1}, which also says  $\mcA_0= \tT_+^\circ$. $\mcA_+ \cap \mcA_- = \mcA_0$ since $\tTz \cap \mcA_0 =\{\zero\}. $

   (b)  By Lemma~\ref{1p1}.
\end{proof}
\begin{rem}
   Clearly $a_3:=a_1+a_2 \in \tT$ for $a_i\in \tT_+$ since otherwise $a_2 =(-)a_1 \in \tT_+ \cap \tT_-\subseteq \mcA_0.$ But one could conceivably have $a_3 \in \tT_-$. For example, for $\mcA = \Z/3\Z$ with classical addition, take $\tT_+ =\{\one\}$ and $\tT_- =\{2\}$.
\end{rem}

\begin{example}
 Here are some real metatangible pairs.
\begin{enumerate} \eroman
    \item    The classical real pair $F$ with underlying group $F\setminus \{\zero\},$ where $F$ is a real field.
       \item
    The  hyperpair $( \nsets(\mathcal S ), \nsets(\mathcal S )_0)$ of the hyperfield of signs.
\end{enumerate}
\end{example}

  \subsection{Square roots}$ $

 We say that $x \in \tT $ is a \textbf{square root} of $a\in \tT_+$ if $x^2=a.$
In this case there are two  distinct square roots, denoted $\sqrt{a}\in \tT_+$ and $(-)\sqrt{a}\in \tT_-$. These are the only square roots when $(\mcA,\mcA_0,(-))$ is a $(-)$-domain, since $x^2=y^2$ for $x,y\in \tT$ implies $x(x+y)(-)y(x+y) = x^2(-)y^2 + (xy)^\circ\in \mcA_0,$ so either $x = y$ or $x+y\in \mcA_0,$ i.e., $x=(-)y.$

\begin{lem} $ $
  \begin{enumerate}\eroman
      \item $a^2 \in \tT_+$ for all $a\in \tT.$

       \item    $\tT$ cannot contain $\sqrt{(-)\one}.$
  \end{enumerate}
\end{lem}
\begin{proof} (i) If $a\in \tT_-,$ then $a^2 = ((-)\one a)^2 = ((-) a)^2 \in \tT_+ $.

(ii)    Otherwise, by (i), $i := \sqrt{(-)\one}\in \tT,$ would satisfy  $(-)\one = i^2 \in \tT_- \cap \tT_+$.
\end{proof}

\begin{lem}\eroman Assume that the system $(\mcA,\mcA_0,(-),\preceq)$ satisfies \eqref{pre2}.
    The following are equivalent, for $a,x\in \tT$:
    \begin{enumerate}
        \item $x^2 =a.$
        \item $x$ is a $\preceq$-root of $\la^2 (-)a $.
         \item  $x$ is a factor-root of $\la^2 (-)a $ with respect to any surpassing relation.
             \item  $x$ is a factor-root of $\la^2 (-)a $ with respect to some surpassing relation.
    \end{enumerate}
\end{lem}
\begin{proof}
    (ii)$\Rightarrow (i)$ If $x^2 (-) a \succeq \zero$ then $x^2 = a$ by the uniquely negated property.

    $(i) \Rightarrow (iii)$ We claim $\la^2(-)a \preceq \la^2 + (x(-)x)\la (-) a = (\la(-)x)(\la+x).$

    $(iv)\Rightarrow (ii)$ If $\la^2(-)a \preceq (\la(-)x)(\la+y),$ then $\zero \la \preceq (y(-)x)$, implying $x=y,$ so $ xy = x^2 = a.$
\end{proof}

 \subsection{Adjoining square roots}

\begin{lem}
 \label{adjsq} Suppose the underlying monoid $\tT$ of the pair $(\mcA,\mcA_0)$ is real.   Then for any non-square $a\in \tT_+$ there is an extension $(\mcA, \mcA_0 )_ {\sqrt{a}} $ with real underlying monoid $\ttT = \tT\cup \tT \sqrt{a}$. If $(\mcA,\mcA_0,(-),\preceq)$ is a system then we have a systemic extension $(\mcA,\mcA_0,(-),\preceq)_{\sqrt{a}}.$

 If \eqref{pre2} holds in $(\mcA,\mcA_0,(-),\preceq)$, then it holds in $(\mcA,\mcA_0,(-),\preceq)_{\sqrt{a}}.$

 If $\mcA$ is a semiring, then so is $\mcA_{\sqrt{a}}.$
\end{lem}

\begin{proof} This is essentially a special case of Theorem~\ref{pol1a1}, but
the notation is easier if one is taking square roots.
Namely, one formally defines $ \ttT = \tT\cup \tT x$ with $x^2 = a.$ In other words, $\mcA_{\sqrt{a}} = \mcA + \mcA x$, endowed with componentwise addition and  scalar multiplication $a_1(b_1+b_2x) = a_1b_1+a_1b_2x$ and $a_1x(b_1+b_2x) = aa_1b_2+a_1b_1x$, is a $\ttT$-module in the obvious way, where we define ${\mcA_{\sqrt{a}}}_0 = \mcA_0 +\mcA_0x.$
Clearly $\ttT$ is real, taking ${(\ttT)}_+ = \tT_+ \cup \tT_+x $ and  ${(\ttT)}_- = \tT_- \cup \tT_-x,$ seen to be disjoint by taking components.

Define $(-)(b_1+b_2x)= (-)b_1+(-)b_2x;$ $({\mcA_{\sqrt{a}}} ,{\mcA_{\sqrt{a}}}_0 )$ is seen to be uniquely negated by taking components.

Define the surpassing relation $b_1+b_2x \preceq b_1'+b_2'x$ if and only if $b_1\preceq b_1'$ and  $b_2\preceq b_2'$.

 \eqref{pre2} is verified componentwise.

Finally, if $\mcA$ is a semiring, we define multiplication $(b_1+b_2x)(b_1'+b_2'x) = b_1b_1'+ b_2b_2'a + (b_1b_2'+b_2b_1')x.$
\end{proof}

\begin{example}
    \label{fs2} Let us determine the condition under which $(\mcA,\mcA_0,(-),\preceq)_{\sqrt{a}}$ is a $(-)$-domain.
    We write $a_1x^i h(-) a_2x^j h \in \mcA_0[\sqrt{a}],$ for $a_1x^i\ne a_2x^j\in \tT_{\sqrt{a}}$ and $0\le i,j\le 1.$ If $i=j$ then we have
    $x^i(a_1h(-) a_2h )\in \mcA_0[\sqrt{a}]$, implying $ a_1h(-) a_2h \in \mcA_0,$ and thus $h \in \mcA_0[\sqrt{a}].$ Hence, we may assume that $i=1$ and $j=0,$ and write $h= b_0+b_1x$. Matching components of  $x$ and $\one$  yields $a_1b_0 (-)a_2b_1\in \mcA_0$ and $aa_1 b_1 (-) a_2b_0\in \mcA_0.$ So the condition required for $(\mcA,\mcA_0,(-),\preceq)_{\sqrt{a}}$ to be a $(-)$-domain is:

     $a_1b_0 (-)a_2b_1\in \mcA_0$ and $aa_1 b_1 (-) a_2b_0\in \mcA_0$  for $a_1,a_2$ in $\tT$ and $b_i\in \mcA$ imply $b_0,b_1\in \mcA_0$.

     In the classical situation, we would subtract and get $(a(-)\one )a_1b_1\in \mcA_0,$ implying $b_1\in \mcA_0$ and likewise $(a(-)\one )a_2b_0\in \mcA_0,$ implying $b_0\in \mcA_0.$
\end{example}

\begin{rem}
    Mimicking the construction in Theorem~\ref{pol1a2} we could expand ${\mcA_{\sqrt{a}}}_0$ to $\widetilde{{\mcA_{\sqrt{a}}}_0}$ defined inductively by setting $\widetilde{{\mcA_{\sqrt{a}}}_0}=\bigcup_k{\mcA_{\sqrt{a};k}} $ where $\mcA_{\sqrt{a};0} = {\mcA_{\sqrt{a}}}_0$ and $$\mcA_{\sqrt{a};k} = \{ h \in \mcA_{\sqrt{a}}: a_1x^ih(-)a_2x^jh \in \mcA_{\sqrt{a};k-1} \text { for suitable } a_1,a_2\in \tT,\ i,j\in \{0,1\}\}.$$ When is $(\mcA_{\sqrt{a}},\widetilde{{\mcA_{\sqrt{a}}}_0})$ a $(-)$-domain?
\end{rem}

\begin{definition}
A tangibly real pair  $(\mcA,\mcA_0)$ is \textbf{quadratically closed} if 
    every element    $a\in\tT_+$ has a square root $\sqrt{a}$ in $\tT_+$.
\end{definition}

\begin{proposition}\label{qc}
    Any tangibly real system  $(\mcA,\mcA_0,(-),\preceq)$ has a \textbf{quadratic closure}.
\end{proposition}
\begin{proof}
   Iterate Lemma~\ref{adjsq} via transfinite induction, adjoining square roots.
\end{proof}

\section{Appendix: Functionally tangible polynomials}

Here is another way of viewing the polynomial function pair of a pair  in which functions $a_1\la^i(-) a_2\la^j$ are in the underlying monoid, avoiding the need   for $(-)$-domains. Define $ \tT^\natural  = \{\text{tangible }f\in \mcA[\la]: f(a)\in \tT$ for almost all $a\in \tT\}$, and    $\tT^\sharp$ to be the set of finite products of elements of $\tT^\natural$. (In particular $a\la^m \in  \tT^\natural $ for all $a\in \tT$ and all $m$.)

\begin{lem}\label{tanp}
    If $f \in \tT^\natural,$  then for any $g\in \mcA[\la],$ $(fg)(a) = f(a)g(a)$ whenever $a\in \tT$ and $f(a)\in \tT$.
\end{lem}
\begin{proof}
    Write $f = \sum_i a_i \lambda^i$ and $g = \sum_j b_j \lambda^j,$ so $fg = \sum _{i,j}a_ib_j \lambda^{i+j}.$ Then $ (fg)(a)=  \sum _{i,j}a_ib_ja^ia^j = \sum _{j} (\sum_i a_i a^i)b_j a^j = \sum _{j} f(a) b_j a^j =  f(a) \sum _{j} b_j a^j = f(a)g(a).$
\end{proof}

\begin{proposition}
If $g\in \tT^\sharp$ then $g(a)\in \tT$ for almost all $a\in \tT$.
\end{proposition}
\begin{proof}
    Suppose $g= f_1\dots f_m$ with $f_1, \dots, f_m \in \tT^\natural$. If $m=1$ then $g\in \tT^\natural$, and there is nothing to prove. Inductively, we may assume that the assertion holds for $h:= f_2\dots f_{m}$, and then by Lemma~\ref{tanp}, $g(a) =f_1(a) h(a)\in \tT$ for those $a\in \tT$ such that $f_1(a)\in \tT$ and $h(a)\in \tT.$
\end{proof}

  \begin{rem}
If $(\mcA  , \mcA_0)$  is  metatangible, then every tangible polynomial is in~$ \tT^\natural ,$ by  Proposition~\ref{HypS}, and $\tT^\sharp$ contains all products of binomials.
    \end{rem}
\begin{theorem}
  We use the equivalence $\equiv$ of Definition~\ref{null1}, writing $[f]$ for the equivalence class of $f.$ Let $\tilde \mcA = \mcA[\la]/\equiv$ and $$\tilde \mcA_0= \{[h] \in \tilde \mcA : gh \in \mcA[\la]_0 \text{ for some } g\in \tT^\sharp\}.$$

  As in Lemma~\ref{ess1}, we have:
\begin{enumerate}\eroman
 \item   $(\tilde \mcA , \tilde \mcA_0)$ is a pair  with respect to the underlying monoid $ \tT^\sharp/\equiv $, under the action $([f_1]\dots [f_m])[g] = [f_1(f_2(\dots f_m g))]$ given in MNote~\ref{nonass}.

 \item  If  $(\mcA , \mcA _0) $ is uniquely negated, then $(\tilde \mcA , \tilde \mcA_0)$ is uniquely negated.

\item If $(\mcA  , \mcA_0)$ has a surpassing relation  $\preceq$, then $(\tilde \mcA , \tilde \mcA_0)$  has a surpassing relation $\preceq_\equiv$ induced by $\prect.$
    \end{enumerate}
\end{theorem}
\begin{proof}
  (i)
 We need to show that    $([f_1]\dots [f_m])[g]$ is well-defined for $f_i\in \tT^\natural.$
 But by Lemma~\ref{tanp}, for almost all $a\in \tT,$
  $$((f_1\dots f_m)g)(a) = (f_1(a)f_2(a)\dots f_m(a)) g(a) = (f_1\dots f_m)(a)g(a).$$

$\tilde \mcA_0$ is a submodule, since if $g_1h_1 \in \mcA[\la]_0$ and $g_2h_2 \in \mcA[\la]_0,$ then $g_1g_2(h_1+h_2) \in \mcA[\la]_0.$

 If $h g \in \mcA[\la]_0$ for $h\in \tT^\sharp$ then for almost every $a$ such that $h(a)\in \tT$ we have $g(a)\in \mcA_0,$ implying $g \in \mcA[\la]_0.$

(ii) Suppose $g(h_1+h_2)\in \mcA[\la]_0$ for $g,h_1,h_2\in \tT^\sharp.$ If $g(a)h_1(a)+g(a)h_2(a)\in \mcA_0$ for $g(a), h_1(a), h_2(a)\in \tT,$ then $h_2(a) =(-)h_1(a);$ since this holds for almost all $a\in \tT,$ we conclude $[h_2]=(-)[h_1].$

(iii) $\preceq_\equiv$ is seen to be a preorder by checking tangible evaluations. If $g\in \mcA[\la]_0$ then $\zero \preceq g(a)$ for almost all $a\in \tT,$ implying $\zero \preceq_\equiv [g].$  If $[g_1]\preceq_\equiv [g_2]$ for $g_i\in  \tT^\sharp ,$ then $g_1(a)\preceq g_2(a)$ for almost all $a\in \tT,$ so $g_1(a)=g_2(a),$ implying $[g_1]= [g_2].$
\end{proof}

We could continue our study of roots of polynomials in $(\tilde \mcA , \tilde \mcA_0),$ a more aesthetically satisfying formulation. An element $a\in \tT$ is a   factor-root of $[f]$   if there is a tangible polynomial $g$ for which $f(\la)\prect  (\la (-)a)g(\la) .$
Although this definition does not depend on the representative $f$ of $[f],$ a \textbf{null root} of $[f]$ must be defined as a common root of all polynomials equivalent to $f.$ Thus the results of \S\ref{ubiquity1} become relevant.

In this case, as in Lemma~\ref{indst0}, when   $(\mcA,\mcA_0,(-))$ is a $(-)$-domain  satisfying  Hypothesis~R, the following analogue of the Root Condition holds:

If \([f]\prect(\lambda(-)a_1)[g]\) and \(a\ne a_1\) is a null root of \([f]\), then \(a\) is a null root of~\([g]\).

Hence   the analogue of Lemma~\ref{indst} holds, and we can obtain a linear $\prect$-splitting with respect to $h=\one$ when $[f]$ has $\deg f$ factor-roots.

\end{document}